\documentclass[a4paper, oneside]{article}
\usepackage{enumitem}
\usepackage{lmodern}
\usepackage{amssymb,amsthm,amsmath}
\usepackage[utf8]{inputenc}
\usepackage{verbatim}
\usepackage{tikz}
\usepackage{authblk}
\usepackage{cite}
\usetikzlibrary{arrows}
\usetikzlibrary{decorations.pathreplacing}

\theoremstyle{definition}
\newtheorem{definition}{Definition}[section]
\newtheorem{lemma}[definition]{Lemma}
\newtheorem{proposition}[definition]{Proposition}
\newtheorem{theorem}[definition]{Theorem}
\newtheorem{example}[definition]{Example}
\newtheorem{corollary}[definition]{Corollary}
\newtheorem{remark}[definition]{Remark}
\newtheorem{problem}[definition]{Problem}

\newcommand{\N}{\mathbb{N}}
\newcommand{\Npos}{{\mathbb{N}_+}}
\newcommand{\R}{\mathbb{R}}
\newcommand{\Z}{\mathbb{Z}}
\newcommand{\C}{\mathbb{C}}
\newcommand{\Rpos}{\mathbb{R}_{>0}}

\newcommand{\abs}[1]{{\left\vert #1 \right\vert}}

\newcommand{\digs}{\Sigma} 
\newcommand{\Mul}{\Pi} 
\newcommand{\mul}{g} 
\newcommand{\fmul}{f} 
\newcommand{\trsh}{\Xi} 
\newcommand{\ind}[1]{{\mathcal{#1}}} 
\newcommand{\glue}{\otimes} 

\DeclareMathOperator{\fractional}{frac} 
\DeclareMathOperator{\config}{config} 
\DeclareMathOperator{\real}{real} 
\DeclareMathOperator{\cyl}{Cyl} 
\DeclareMathOperator{\tr}{Tr} 
\DeclareMathOperator{\pre}{pred} 
\DeclareMathOperator{\md}{Md} 

\begin{document}

\title{On the Trace Subshifts of Fractional Multiplication Automata}

\author{Johan Kopra}

\affil{Department of Mathematics and Statistics, \\FI-20014 University of Turku, Finland}
\affil{jtjkop@utu.fi}

\date{}

\maketitle

\setcounter{page}{1}

\begin{abstract}We address the dynamics of the cellular automaton (CA) that multiplies by $p/q$ in base $pq$ (for coprime $p>q>1$) by studying its trace subshift. We present a conjugacy of the trace to a previously studied base-$p/q$ numeration system. We also show that the trace subshift is not synchronizing and in particular not sofic. As a byproduct we compute its complexity function and we conclude by presenting an example of a sofic shift with the same complexity function.  \end{abstract}

\providecommand{\keywords}[1]{\textbf{Keywords:} #1}
\noindent\keywords{cellular automata, sofic subshifts, trace subshifts}

\section*{Introduction}

A cellular automaton (CA) is a model of parallel computation consisting of a uniform (in our case one-dimensional) grid of finite state machines, each of which receives input from a finite number of neighbors. All the machines use the same local update rule to update their states simultaneously at discrete time steps. An interesting natural class of automata is given by multiplication automata $\Mul_{p/q,pq}:\digs_{pq}^\Z\to \digs_{pq}^\Z$ (with coprime $p>q>1$) acting on bi-infinite sequences (configurations) over the digit set $\digs_{pq}=\{0,1,\dots,pq-1\}$, which perform multiplication by $p/q$ on base-$pq$ representations of nonnegative real numbers. Figure~\ref{frac} shows the elements $x,\Mul_{3/2,6}(x),\Mul_{3/2,6}^2(x),\dots$ on consecutive rows for $x\in\digs_{pq}^\Z$ representing the number $1$. Such a figure containing repeated applications of some CA $F$ on a configuration $x$ is called the \emph{space-time diagram} of $x$ (with respect to $F$). 

\begin{figure}[p]
\centering
\includegraphics[scale=0.5]{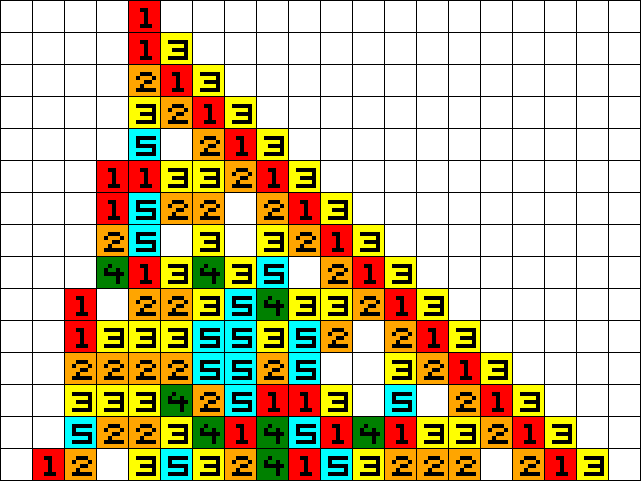}
\caption{Repeated multiplication of the number $1$ by the multiplication automaton $\Mul_{3/2,6}$. The $0$-digit is denoted by a white square.}
\label{frac}
\end{figure}

\begin{figure}[p]
\centering
\includegraphics[scale=0.48]{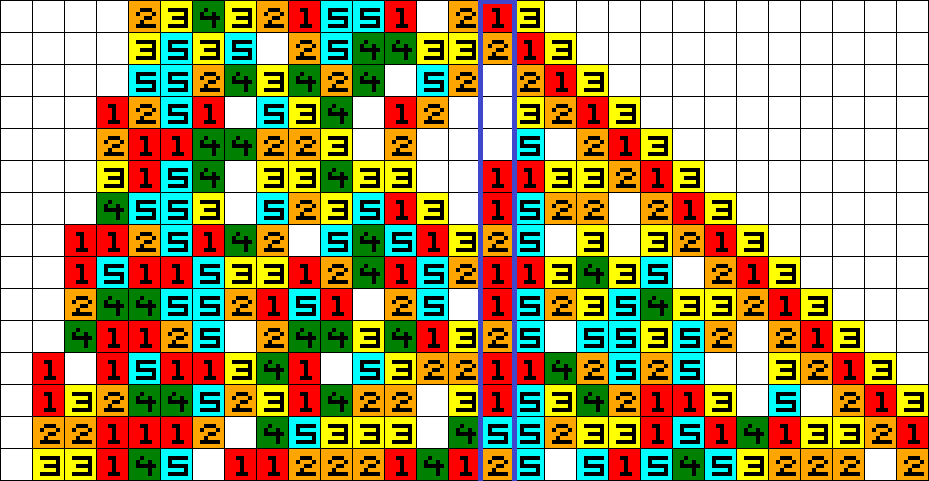}
\caption{A space-time diagram of a configuration $x$ (with respect to $\Mul_{3/2,6}$) that initially seems to be a base-$6$ representation of a $Z$-number. Note however an occurrence of the digit $5$ on the fourteenth row in the highlighted column.}
\label{fracMahler}
\end{figure}

A connection of these automata to Mahler's $3/2$-problem was noted in~\cite{Kari12a}. Mahler's $3/2$-problem~\cite{Mah68} is to determine whether there exists a $Z$-number, i.e. a number $\xi>0$ such that the fractional part of $\left(\frac{3}{2}\right)^n\xi$ is less than $\frac{1}{2}$ for all $n\in\N$. This means that for all $n\in\N$ the base-$6$ expansion of $\left(\frac{3}{2}\right)^n\xi$ contains $0,1$ or $2$ to the right of the decimal point. Therefore the existence of a $Z$-number is equivalent to the existence of a configuration $x\in\digs_6^\Z$ that contains only zeroes sufficiently far to the left and whose space-time diagram contains only digits $0,1$ and $2$ down some column (see an ``almost'' $Z$-number in Figure~\ref{fracMahler}). Any infinite sequence that can appear in the space-time diagram of some given CA is called a \emph{trace} and the collection of all possible traces is the \emph{trace subshift}.

In this paper we study the trace subshifts of multiplication automata $\Mul_{p/q,pq}$. One motivation for their study is the connection to Mahler's $3/2$-problem (and its natural generalization studied e.g. in~\cite{FLP95}, where we ask about the existence of a number  $\xi>0$ such that the fractional part of $\left(\frac{p}{q}\right)^n\xi$ is less than $\frac{1}{q}$ for all $n\in\N$) presented in the previous paragraph. Our main result concerns the computational tractability of the trace subshift. There is a trivial upper bound for the complexity: the language of the trace subshift of any CA has to be recursive. When it comes to the lower bound, it would be particularly nice if the trace shifts of $\Mul_{p/q,pq}$ were sofic, because then they could be represented as the collection of labels of infinite paths on some finite graph. We will show in Theorem~\ref{notsoficMahler} that the restriction of the trace shift of $\Mul_{p/q,pq}$ to the symbol set $\digs_p=\{0,\dots,p-1\}$ is not sofic and in Theorem~\ref{notsynch} that the trace shift in itself is not even synchronizing. Either of these results imply that $\Mul_{p/q,pq}$ is not a regular CA in K\r{u}rka's classification~\cite{Kur95}. 

In addition to proving our main result we make some other notes on the trace subshift. We present a conjugacy between the trace subshift and a previously studied base-$pq$ numeration system. We also compute the complexity function of the trace shift and the corresponding generating function. 

Large parts of this paper have appeared previously in the author's doctoral dissertation~\cite{KopDiss}. Section~\ref{secRealize} is completely new. Its purpose is to show that only studying the complexity function of the trace shift of $\Mul_{p/q,pq}$ is not sufficient to conclude that it is not sofic.

\section{Preliminaries}\label{sectPreli}

It is natural to consider cellular automata and subshifts in the context of general topological dynamics. Standard references for topological and symbolic dynamics are~\cite{Kur03} and~\cite{LM95}.

\begin{definition}If $X$ is a compact metrizable topological space and $T:X\to X$ is a continuous map, we say that $(X,T)$ is a \emph{(topological) dynamical system}.\end{definition}

\begin{definition}The dynamical system $(X,T)$ is \emph{transitive}, if for all nonempty open sets $U,V\subseteq X$ there exists $n\in\N$ such that $T^n(U)\cap V\neq\emptyset$.
\end{definition}

When there is no risk of confusion, we may identify the dynamical system $(X,T)$ with the underlying space or the underlying map, so we may say that $X$ is a dynamical system or that $T$ is a dynamical system.

The structure preserving transformations between topological dynamical systems are known as morphisms.

\begin{definition}We write $\psi:(X,T)\to (Y,S)$ whenever $(X,T)$ and $(Y,S)$ are dynamical systems and $\psi:X\to Y$ is a continuous map such that $\psi\circ T=S\circ\psi$. Then we say that $\psi$ is a \emph{morphism}. If $\psi$ is injective, we say that $\psi$ is an \emph{embedding}. If $\psi$ is surjective, we say that $\psi$ is a \emph{factor map} and that $(Y,S)$ is a factor of $(X,T)$ (via the map $\psi$). If $\psi$ is bijective, we say that $\psi$ is a \emph{conjugacy} and that $(X,T)$ and $(Y,S)$ are \emph{conjugate} (via $\psi$).\end{definition}

A finite set $A$ containing at least two elements (\emph{letters}) is called an \emph{alphabet}. Occasionally we want the alphabet to consist of numbers and thus for $n\in\Npos$ we denote $\digs_n=\{0,1,\dots,n-1\}$. The set $A^\Z$ of bi-infinite sequences (\emph{configurations}) over $A$ is called a \emph{full shift}. Formally any $x\in A^\Z$ is a function $\Z\to A$ and the value of $x$ at $i\in\Z$ is denoted by $x[i]$. It contains finite, right-infinite and left-infinite subsequences denoted by $x[i,j]=x[i]x[i+1]\cdots x[j]$, $x[i,\infty]=x[i]x[i+1]\cdots$ and $x[-\infty,i]=\cdots x[i-1]x[i]$. Occasionally we signify the symbol at position zero in a configuration $x$ by a dot as follows:
\[x=\cdots x[-2]x[-1]x[0].x[1]x[2]x[3]\cdots.\]

A configuration $x\in A^\Z$ is \emph{periodic} if there is a $p\in\Npos$ such that $x[i+p]=x[i]$ for all $i\in\Z$. Then we may also say that $x$ is $p$-periodic or that $x$ has period $p$. If $x$ is not periodic, it is \emph{aperiodic}. We say that $x$ is \emph{eventually periodic} if there are $p\in\Npos$ and $i_0\in\Z$ such that $x[i+p]=x[i]$ holds for all $i\geq i_0$.

A \emph{subword} of $x\in A^\Z$ is any finite sequence $x[i,j]$ where $i,j\in\Z$, and we interpret the sequence to be empty if $j<i$. Any finite sequence $w=w[1] w[2]\cdots w[n]$ (also the empty sequence, which is denoted by $\epsilon$) where $w[i]\in A$ is a \emph{word} over $A$. Unless we consider a word $w$ as a subword of some configuration, we start indexing the symbols of $w$ from $1$ as we have done here. The concatenation of a word or a left-infinite sequence $u$ with a word or a right-infinite sequence $v$ is denoted by $uv$. A word $u$ is a \emph{prefix} of a word or a right-infinite sequence $x$ if there is a word or a right-infinite sequence $v$ such that $x=uv$. Similarly, $u$ is a \emph{suffix} of a word or a left-infinite sequence $x$ if there is a word or a left-infinite sequence $v$ such that $x=vu$. The set of all words over $A$ is denoted by $A^*$, and the set of non-empty words is $A^+=A^*\setminus\{\epsilon\}$. The set of words of length $n$ is denoted by $A^n$. For a word $w\in A^*$, $\abs{w}$ denotes its length, i.e. $\abs{w}=n\iff w\in A^n$. For any word $w\in A^+$ we denote by ${}^\infty w$ and $w^\infty$ the left- and right-infinite sequences obtained by infinite repetitions of the word $w$. We denote by $w^\Z\in A^\Z$ the configuration defined by $w^\Z[in,(i+1)n-1]=w$ (where $n=\abs{w}$) for every $i\in\Z$. In the full shift $\digs_n^\Z$ we say that $x\in\digs_n^\Z$ is finite if $x[-\infty,i]={}^\infty 0$ and $x[j,\infty]=0^\infty$ for some $i,j\in\Z$.

Any collection of words $L\subseteq A^*$ is called a \emph{language}. For any $S\subseteq A^\Z$ the collection of words appearing as subwords of elements of $S$ is the language of $S$, denoted by $L(S)$. For $n\in\N$ we denote $L^n(S)=L(S)\cap A^n$. The \emph{complexity function} of $S$ is the map $P_S:\N\to\N$ defined by $P_S(n)=\abs{L^n(S)}$ for $n\in\N$. For any $L,K\subseteq A^*$, let
\[LK=\{uv\mid u\in L, v\in K\},\quad L^*=\{w_1\cdots w_n\mid n\geq 0,w_i\in L\}\subseteq A^*,\]
i.e. $L^*$ is the set of all finite concatenations of elements of $L$. If $\epsilon\notin L$, define $L^+=L^*\setminus\{\epsilon\}$ and if $\epsilon\in L$, define $L^+=L^*$.

For $x,y\in A^\Z$ and $i\in\Z$ we denote by $x\glue_i y\in A^\Z$ the ``gluing'' of $x$ and $y$ at $i$, i.e. $(x\glue_i y)[-\infty,i-1]=x[-\infty,i-1]$ and $(x\glue_i y)[i,\infty]=y[i,\infty]$. Typically we perform gluings at the origin and we denote $x\glue y=x\glue_0y$.

To consider topological dynamics on subsets of the full shift, the set $A^\Z$ is endowed with the product topology (with respect to the discrete topology on $A$). The shift map $\sigma:A^\Z\to A^\Z$ is defined by $\sigma(x)[i]=x[i+1]$ for $x\in A^\Z$, $i\in\Z$, and it is a homeomorphism. Any topologically closed \emph{nonempty} subset $X\subseteq A^\Z$ such that $\sigma(X)=X$ is called a \emph{subshift}. It is also a compact metrizable space under the subspace topology induced from $A^\Z$. The restriction of $\sigma$ to $X$ is also a homeomorphism and it may be denoted by $\sigma_X$. Typically the subscript $X$ is omitted from all notations when $X$ is clear from the context. Every subshift $X$ is identified with the dynamical system $(X,\sigma)$ induced by the shift map $\sigma$. For subshifts $X$ an alternative characterization of transitivity is that for all words $u,v\in L(X)$ there is a word $w\in L(X)$ such that $uwv\in L(X)$. 

\begin{definition}
A language $L\subseteq A^\Z$ is \emph{factorial} if for every $w\in L$ and every subword $v$ of $w$ it also holds that $w\in L$. It is \emph{extendable} if for every $w\in L$ there are $a,b\in A$ such that $awb\in L$.
\end{definition}
Whenever $L$ is a factorial extendable language, there is a subshift $X$ such that $L(X)=L$.

In this paper we consider two particular classes of subshifts.

\begin{definition}
A subshift $X$ is a \emph{sofic shift} if $L(X)$ is regular language.
\end{definition}

By an alternative characterization, a subshift is sofic if and only if it is a factor of a subshift of finite type. In particular it follows that any factor of a sofic subshift is also sofic.

\begin{definition}
Given a subshift $X$, we say that a word $w\in L(X)$ is synchronizing if
\[\forall u,v\in L(X): uw,wv\in L(X)\implies uwv\in L(X).\]
We say that a transitive subshift $X$ is synchronizing if $L(X)$ contains a synchronizing word.
\end{definition}
Transitive sofic shifts in particular are synchronizing, which follows by using the results of~\cite{LM95} in Section~3.3 and in Exercise~3.3.3.

Given a subshift $X\subseteq A^\Z$ and a word $w\in L(X)$ we define the set of predecessors of $w$ in $X$ by
\[\pre_X(w)=\{a\in A\mid aw\in L(X)\}.\]
The notion of predecessors is extended to one-way infinite sequences. For $x\in X$ we define
\begin{flalign*}
\pre_X(x[0,\infty])=\bigcap_{n\in\N}\pre_X(x[0,n]).
\end{flalign*}

\begin{definition}
Let $X\subseteq A^\Z$ and $Y\subseteq B^\Z$ be subshifts. We say that the map $F:X\to Y$ is a \emph{sliding block code} from $X$ to $Y$ (with memory $m$ and anticipation $a$ for integers $m\leq a$) if there exists a \emph{local rule} $f:A^{a-m+1}\to B$ such that $F(x)[i]=f(x[i+m],\dots,x[i],\dots,x[i+a])$. If $X=Y$, we say that $F$ is a \emph{cellular automaton} (CA). If we can choose $m$ and $a$ so that $-m=a=r\geq 0$, we say that $F$ is a radius-$r$ CA.
\end{definition}

Note that both memory and anticipation can be either positive or negative. Note also that if $F$ has memory $m$ and anticipation $a$ with the associated local rule $f:A^{a-m+1}\to A$, then $F$ is also a radius-$r$ CA for $r=\max\{\abs{m},\abs{a}\}$, with possibly a different local rule $f':A^{2r+1}\to A$. We can extend any local rule $f:A^{d+1}\to B$ (where $d=a-m$) to words $w=w[1]\cdots w[d+n]\in A^{d+n}$ with $n\in\Npos$ by $f(w)=u=u[1]\cdots u[n]$, where $u[i]= f(w[i],\dots,\allowbreak w[i+d])$.

Sliding block codes are morphisms between subshifts, and vice versa~\cite{Hed69}, and bijective sliding block codes are conjugacies. Bijective CA are called \emph{reversible}. It is known that the inverse map of a reversible CA is also a CA.

For a subshift $X\subseteq A^\Z$, a reversible CA $F:X\to X$, a configuration $x\in X$ and a nonempty interval $I=[i,j]\subseteq \Z$, the $I$-\emph{trace} of $x$ (with respect to $F$) is the configuration $\tr_{F,I}(x)$ over the alphabet $A^{\abs{I}}$ defined by
\[\tr_{F,I}(x)[t]=(F^t(x)[i],F^t(x)[i+1],\dots,F^t(x)[j]) \mbox{ for } t\in\Z.\]
If $I=\{i\}$ is the degenerate interval, we may write $\tr_{F,i}(x)$ and if $i=0$, we may write $\tr_F(x)$. If the CA $F$ is clear from the context, we may write $\tr_I(x)$. The $I$-\emph{trace subshift of} $F$ is defined by
\[\trsh_I(F)=\tr_{F,I}(X)\subseteq (A^\abs{I})^\Z,\]
This is indeed a subshift. Namely, $\trsh_I(F)$ is closed in $(A^\abs{I})^\Z$ as the image of the compact set $X$ under the continuous map $\tr_{F,I}$. It is also closed under $\sigma$, because any $z\in \trsh_I(F)$ has a preimage $x\in X$ and then the image of $F(x)$ by $\tr_{F,I}$ is $\sigma(z)$. This argument also shows that $\tr_{F,I}:(X,F)\to (\trsh_I(F),\sigma)$ is a factor map. We may omit the subscript if $I=\{0\}$, i.e. $\trsh(F)=\trsh_{\{0\}}(F)$.

The trace subshifts of $F$ form a universal collection of subshift factors of the dynamical system $(X,F)$ in the sense that any factor map $\psi:(X,F)\to (Z,\sigma)$ to a subshift $(Z,\sigma)$ factors through a trace subshift, i.e. there is an interval $I\subseteq \Z$ and a factor map $\psi':(\trsh_I(F),\sigma)\to(Z,\sigma)$ such that $\psi=\psi'\circ\tr_{F,I}$.

K\r{u}rka suggested a language theoretical classification for cellular automata. The following definition was given in~\cite{Kur97} for general dynamical systems on zero-dimensional spaces.

\begin{definition}
A cellular automaton $F:X\to X$ is \emph{regular} if all its subshift factors are sofic shifts.
\end{definition}

This definition is motivated in~\cite{Kur95}. Taking a subshift factor $Y$ of $F:X\to X$ corresponds to taking a finite (clopen) partition $\{X_1,\dots, X_n\}$ of $X$, an ``observation window'', and observing for each $x\in X$ the infinite sequence of partition elements visited by $x$ under repeated application of the map $F$. Regularity of $F$ means that the totality of all sequences of observations form a ``simple'' set $Y$ for arbitrarily precise observation windows. On the other hand, non-regularity means that $F$ has complex behavior that can be detected by a suitable partition of $X$. Since all subshift factors of sofic subshifts are sofic, and since the trace subshifts of $F$ form a universal collection of subshift factors for $(X,F)$, to test the regularity of $F$ it is sufficient to test the soficness of the trace subshifts.

\section{Multiplication automata}

In this section we introduce the fractional multiplication automata $\Mul_{p/q,pq}$ multiplying by $p/q$ in base $pq$. We begin by giving a natural definition of what it means for a cellular automaton to perform multiplication by nonnegative numbers. Such automata have been considered earlier in \cite{BHM96,BM97,Har12,Kari12a,Kari12b,KK17,Rud90}. Then we present the construction of $\Mul_{p/q,pq}$ and, after restricting to the case of coprime $p,q>1$ starting from Remark~\ref{coprimeAss}, prove some basic properties of multiplication automata. Some of the lemmas of this section have appeared previously in~\cite{KK17}.

Recall that $\digs_n=\{0,1,\dots,n-1\}$ for $n\in\N$, $n>1$. To perform multiplication using a CA we need be able to represent a nonnegative real number as a configuration in $\digs_n^{\Z}$. If $\xi\geq0$ is a real number and $\xi=\sum_{i=-\infty}^{\infty}{\xi_i n^i}$ is the unique base-$n$ expansion of $\xi$ such that $\xi_i\neq n-1$ for infinitely many $i<0$, we define $\config_n(\xi)\in \digs_n^{\Z}$ by
\[\config_n(\xi)[i]=\xi_{-i}\]
for all $i\in\Z$. In reverse, whenever $x\in \digs_n^{\Z}$ is such that $x[i]=0$ for all sufficiently small $i$, we define
\[\real_n(x)=\sum_{i=-\infty}^{\infty}{x[-i] n^i}.\]
Clearly $\real_n(\config_n(\xi))=\xi$ and $\config_n(\real_n(x))=x$ for every $\xi\geq0$ and every $x\in \digs_n^{\Z}$ such that $x[i]=0$ for all sufficiently small $i$ and $x[i]\neq n-1$ for infinitely many $i>0$.

The fractional part of a number $\xi\in\R$ is
\[\fractional(\xi)=\xi-\lfloor \xi\rfloor\in[0,1).\]

\begin{definition}\label{multdef}For $\alpha\in\Rpos$ and a natural number $n\geq2$, we denote by $\Mul_{\alpha,n}:\digs_n^\Z\to \digs_n^\Z$ the cellular automaton such that
\[\real(\Mul_{\alpha,n}(x))=\alpha\real(x)\]
for every finite configuration $x\in \digs_n^\Z$, whenever such an automaton exists. We say that $\Mul_{\alpha,n}$ multiplies by $\alpha$ in base $n$.
\end{definition}

The cellular automaton of this definition is unique whenever it exists. To see this, let $F$ and $F'$ be CA that satisfy the assumption for some $\alpha,n$. The function $\real:\digs_n^\Z\to \R$ is clearly injective on the set of finite configurations, so the values of $F$ and $F'$ are determined on the dense set of finite configurations. Since $F$ and $F'$ are continuous functions that agree on a dense set, it follows that $F=F'$. We note that in~\cite{BHM96} the possible pairs $\alpha,n$ have been characterized for one-sided configuration spaces $\digs_n^\N$. A characterization could also be given in the case $\digs_n^\Z$ that we consider along the same lines as in~\cite{BHM96} or by an alternative method of Section~3.2 in~\cite{KopDiss}.

For integers $p,n\geq2$ where $p$ divides $n$ let $\mul_{p,n}:\digs_{n}\times \digs_{n}\to \digs_{n}$ be defined as follows. Let $q$ be such that $pq=n$. Digits $a,b\in \digs_{pq}$ are represented as $a=a_1q+a_0$ and $b=b_1q+b_0$, where $a_0,b_0\in \digs_q$ and $a_1,b_1\in \digs_p$: such representations always exist and they are unique. Then
\[\mul_{p,n}(a,b)=\mul_{p,n}(a_1q+a_0,b_1q+b_0)=a_0p+b_1.\]
An example in the particular case $(p,n)=(3,6)$ is given in Figure \ref{taul}.

\begin{figure}[ht]
\centering
\begin{tabular} {c | c c c c c c}
$a\backslash b$ & 0 & 1 & 2 & 3 & 4 & 5 \\ \hline
0 & 0 & 0 & 1 & 1 & 2 & 2 \\
1 & 3 & 3 & 4 & 4 & 5 & 5 \\
2 & 0 & 0 & 1 & 1 & 2 & 2 \\
3 & 3 & 3 & 4 & 4 & 5 & 5 \\
4 & 0 & 0 & 1 & 1 & 2 & 2 \\
5 & 3 & 3 & 4 & 4 & 5 & 5 \\
\end{tabular}
\caption{The values of $\mul_{3,6}(a,b)$.}
\label{taul}
\end{figure}

We define the CA $\Mul_{p,n}:\digs_{n}^{\Z}\to \digs_{n}^{\Z}$ by $\Mul_{p,n}(x)[i]=\mul_{p,n}(x[i],x[i+1])$, so $\Mul_{p,n}$ has memory $0$ and anticipation $1$. Giving the name $\Mul_{p,n}$ to this CA is in agreement with Definition~\ref{multdef} by the following lemma.

\begin{lemma}\label{vastaavuus}$\real_{n}(\Mul_{p,n}(\config_{n}(\xi)))=p\xi$ for all $\xi\geq 0$.\end{lemma}
\begin{proof}
Let $x=\config_{n}(\xi)$. Let $pq=n$ and for every $i\in\Z$, denote by $x[i]_0$ and $x[i]_1$ the natural numbers such that $0\leq x[i]_0<q$, $0\leq x[i]_1<p$ and $x[i]=x[i]_1q+x[i]_0$. Then
\begin{flalign*}
&\real_{n}(\Mul_{p,n}(\config_{n}(\xi)))=\real_{n}(\Mul_{p,n}(x))=\sum_{i=-\infty}^{\infty}\Mul_{p,n}(x)[-i](pq)^i \\
&=\sum_{i=-\infty}^{\infty}\mul_{p,n}(x[-i],x[-i+1])(pq)^i=\sum_{i=-\infty}^{\infty}(x[-i]_0 p+x[-i+1]_1)(pq)^i \\
&=\sum_{i=-\infty}^{\infty}(x[-i]_0 p(pq)^i+x[-i+1]_1pq(pq)^{i-1}) \\
&=\sum_{i=-\infty}^{\infty}(x[-i]_0 p(pq)^i+x[-i]_1pq(pq)^i) \\
&=p\sum_{i=-\infty}^{\infty}(x[-i]_1q+x[-i]_0)(pq)^i=p\real_{pq}(x)=p\real_{pq}(\config_{pq}(\xi))=p\xi.
\end{flalign*}
\end{proof}

We have now seen that the CA $\Mul_{p,n}$ and $\Mul_{q,n}$ exist when $p,q\in\N$ are such that $pq=n$. We show that in this case $\Mul_{p,n}$ is reversible. Indeed, if $x\in \digs_{n}^{\Z}$ is a configuration with a finite number of non-zero coordinates, then
\begin{flalign*}
&\Mul_{q,n}(\Mul_{p,n}(x))=\Mul_{q,n}(\Mul_{p,n}(\config_{pq}(\real_{pq}(x)))) \\
\overset{L \ref{vastaavuus}}{=}& \Mul_{q,n}(\config_{pq}(p\real_{pq}(x)))\overset{L \ref{vastaavuus}}{=}\config_{pq}((pq\real_{pq}(x))=\sigma(x).
\end{flalign*}
Since $\sigma^{-1}\circ \Mul_{q,n}\circ \Mul_{p,n}$ is continuous and agrees with the identity function on a dense set, it follows that $\sigma^{-1}(\Mul_{q,n}(\Mul_{p,n}(x)))=x$ for all configurations $x\in \digs_{pq}^{\Z}$. Similarly $\Mul_{p,n}(\sigma^{-1}(\Mul_{q,n}(x)))=x$ for $x\in \digs_{pq}^{\Z}$. Thus $\sigma^{-1}(\Mul_{q,n}(x))$ is the inverse of $\Mul_{p,n}$ and it must be equal to $\Mul_{1/p,n}$.

The shift CA $\sigma:\digs_{pq}^\Z\to\digs_{pq}^\Z$ multiplies by $pq$ in base $pq$ and its inverse divides by $pq$. This combined with Lemma \ref{vastaavuus} shows that the CA $\Mul_{p/q,pq}$ multiplying by $p/q$ in base $pq$ can be constructed as the composition $\sigma^{-1}\circ \Mul_{p,pq}\circ \Mul_{p,pq}$. Earlier we explicitly defined local rules $\mul_{p,pq}$ for the automata $\Mul_{p,pq}$ which we can use to define local rules $\fmul_{p/q,pq}:\digs_{pq}^3\to\digs_{pq}$ also for the automata $\Mul_{p/q,pq}$ as follows:
\begin{flalign*}
\Mul_{p/q,pq}(x)[i]=&\fmul_{p/q,pq}(x[i-1],x[i],x[i+1]) \\
\doteqdot&\mul_{p,pq}(\mul_{p,pq}(x[i-1],x[i]),\mul_{p,pq}(x[i],x[i+1]));
\end{flalign*}
the symbol $f$ in $\fmul_{p/q,pq}$ is used to emphasize the fact that this local rule is associated with multiplication by a \emph{fraction}.

\begin{remark}\label{coprimeAss}
In the rest of this paper we assume that $p,q>1$ are coprime integers unless specified otherwise.
\end{remark}

As an example, the local rule $\fmul_{3/2,6}$ has been written out explicitly in Figure \ref{lokaali32}. We will prove some of the regularities seen in this figure for general $\fmul_{p/q,pq}$.

\begin{figure}[ht]
\centering
\begin{tabular} {l l}
$c=0$ & $c=1$ \\
\begin{tabular} {c | c c c c c c}
$a\backslash b$ & 0 & 1 & 2 & 3 & 4 & 5 \\ \hline
0 & 0 & 0 & 0 & 0 & 1 & 1 \\
1 & 3 & 3 & 3 & 3 & 4 & 4 \\
2 & 0 & 0 & 0 & 0 & 1 & 1 \\
3 & 3 & 3 & 3 & 3 & 4 & 4 \\
4 & 0 & 0 & 0 & 0 & 1 & 1 \\
5 & 3 & 3 & 3 & 3 & 4 & 4 \\
\end{tabular}
&
\begin{tabular} {c | c c c c c c}
$a\backslash b$ & 0 & 1 & 2 & 3 & 4 & 5 \\ \hline
0 & 1 & 1 & 2 & 2 & 2 & 2 \\
1 & 4 & 4 & 5 & 5 & 5 & 5 \\
2 & 1 & 1 & 2 & 2 & 2 & 2 \\
3 & 4 & 4 & 5 & 5 & 5 & 5 \\
4 & 1 & 1 & 2 & 2 & 2 & 2 \\
5 & 4 & 4 & 5 & 5 & 5 & 5 \\
\end{tabular}
\\
\\
$c=2$ & $c=3$ \\
\begin{tabular} {c | c c c c c c}
$a\backslash b$ & 0 & 1 & 2 & 3 & 4 & 5 \\ \hline
0 & 3 & 3 & 3 & 3 & 4 & 4 \\
1 & 0 & 0 & 0 & 0 & 1 & 1 \\
2 & 3 & 3 & 3 & 3 & 4 & 4 \\
3 & 0 & 0 & 0 & 0 & 1 & 1 \\
4 & 3 & 3 & 3 & 3 & 4 & 4 \\
5 & 0 & 0 & 0 & 0 & 1 & 1 \\
\end{tabular}
&
\begin{tabular} {c | c c c c c c}
$a\backslash b$ & 0 & 1 & 2 & 3 & 4 & 5 \\ \hline
0 & 4 & 4 & 5 & 5 & 5 & 5 \\
1 & 1 & 1 & 2 & 2 & 2 & 2 \\
2 & 4 & 4 & 5 & 5 & 5 & 5 \\
3 & 1 & 1 & 2 & 2 & 2 & 2 \\
4 & 4 & 4 & 5 & 5 & 5 & 5 \\
5 & 1 & 1 & 2 & 2 & 2 & 2 \\
\end{tabular}
\\
\\
$c=4$ & $c=5$ \\
\begin{tabular} {c | c c c c c c}
$a\backslash b$ & 0 & 1 & 2 & 3 & 4 & 5 \\ \hline
0 & 0 & 0 & 0 & 0 & 1 & 1 \\
1 & 3 & 3 & 3 & 3 & 4 & 4 \\
2 & 0 & 0 & 0 & 0 & 1 & 1 \\
3 & 3 & 3 & 3 & 3 & 4 & 4 \\
4 & 0 & 0 & 0 & 0 & 1 & 1 \\
5 & 3 & 3 & 3 & 3 & 4 & 4 \\
\end{tabular}
&
\begin{tabular} {c | c c c c c c}
$a\backslash b$ & 0 & 1 & 2 & 3 & 4 & 5 \\ \hline
0 & 1 & 1 & 2 & 2 & 2 & 2 \\
1 & 4 & 4 & 5 & 5 & 5 & 5 \\
2 & 1 & 1 & 2 & 2 & 2 & 2 \\
3 & 4 & 4 & 5 & 5 & 5 & 5 \\
4 & 1 & 1 & 2 & 2 & 2 & 2 \\
5 & 4 & 4 & 5 & 5 & 5 & 5 \\
\end{tabular}
\end{tabular}
\caption{The values of $\fmul_{3/2,6}(a,c,b)$.}
\label{lokaali32}
\end{figure}

By the construction of $\Mul_{p/q,pq}$, for every $x\in \digs_{pq}^{\Z}$ and every $i\in\Z$ the value of $\Mul_{p/q,pq}(x)[i]$ can be computed from  $x[i-1],x[i]$ and $x[i+1]$, the three nearest digits above in the space-time diagram. Proposition \ref{leftdet}, originally proven in~\cite{KK17}, gives similarly that each digit in the space-time diagram can be computed from the three nearest digits to the right (see Figure \ref{expdirs}). We reproduce its proof here for the sake of completeness.

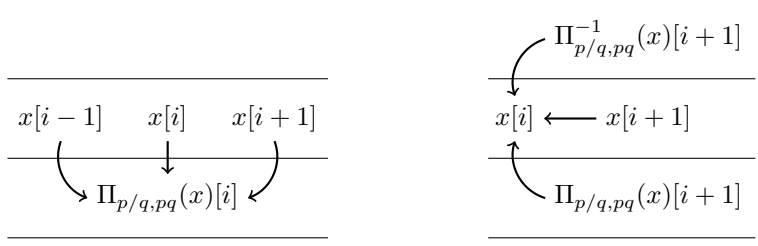
\begin{figure}
\centering
\begin{tikzpicture}
\draw (0pt,60pt) -- (120pt,60pt);
\draw (0pt,30pt) -- (120pt,30pt);
\draw (0pt,0pt) -- (120pt,0pt);
\node (u1) at (20pt,45pt) {$x[i-1]$}; \node (u2) at (60pt,45pt) {$x[i]$}; \node (u3) at (100pt,45pt) {$x[i+1]$};
                                 \node (d) at (60pt,15pt) {$\Mul_{p/q,pq}(x)[i]$};
\draw[->,thick](u1.south) to [bend right = 45] (d.west);\draw[->,thick](u2) to (d);\draw[->,thick](u3.south) to [bend left = 45] (d.east);
																
\draw (180pt,90pt) -- (280pt,90pt);
\draw (180pt,60pt) -- (280pt,60pt);
\draw (180pt,30pt) -- (280pt,30pt);
\draw (180pt,0pt) -- (280pt,0pt);
                                  \node (r1) at (240pt,75pt) {$\Mul_{p/q,pq}^{-1}(x)[i+1]$};
\node (l) at (190pt,45pt) {$x[i]$}; \node (r2) at (240pt,45pt) {$x[i+1]$};
                                  \node (r3) at (240pt,15pt) {$\Mul_{p/q,pq}(x)[i+1]$};
\draw[->,thick](r1.west) to [bend right = 45] (l.north);\draw[->,thick](r2) to (l);\draw[->,thick](r3.west) to [bend left = 45] (l.south);
\end{tikzpicture}
\caption{Determination of digits in the space-time diagram of $x$ with respect to $\Mul_{p/q,pq}$.}
\label{expdirs}
\end{figure}

\begin{lemma}\label{g1}If $\mul_{p,pq}(a,c)=\mul_{p,pq}(b,d)$, then $a\equiv b \pmod q$.\end{lemma}
\begin{proof}Let $a=a_1q+a_0$, $b=b_1q+b_0$, $c=c_1q+c_0$ and $d=d_1q+d_0$. Then
\begin{flalign*}
\mul_{p,pq}(a,c)=\mul_{p,pq}(b,d)&\implies a_0p+c_1=b_0p+d_1 \\
&\implies a_0=b_0\implies a\equiv b\pmod q.
\end{flalign*}
\end{proof}

\begin{lemma}\label{g2}$\mul_{p,pq}(a,c)\equiv \mul_{p,pq}(b,c)\pmod q\iff a\equiv b \pmod q\iff \mul_{p,pq}(a,c)=\mul_{p,pq}(b,c)$.\end{lemma}
\begin{proof}Let $a=a_1q+a_0$, $b=b_1q+b_0$ and $c=c_1q+c_0$. Then
\begin{flalign*}
\mul_{p,pq}(a,c)\equiv \mul_{p,pq}(b,c)\pmod q&\iff a_0p+c_1\equiv b_0p+c_1\pmod q \\
&\iff a_0=b_0\iff a\equiv b\pmod q 
\end{flalign*}
and
\begin{flalign*}
\mul_{p,pq}(a,c)=\mul_{p,pq}(b,c)&\iff a_0p+c_1=b_0p+c_1 \\
&\iff a_0=b_0\iff a\equiv b\pmod q. 
\end{flalign*}
\end{proof}

These basic properties of $\mul_{p,pq}$ can be used to prove the following lemma concerning $\fmul_{p/q,pq}$, because $\fmul_{p/q,pq}$ was defined using $\mul_{p,pq}$. Similar reductions of $\fmul_{p/q,pq}$ to $\mul_{p,pq}$ will be done also later.

\begin{lemma}\label{f1}If $\fmul_{p/q,pq}(a,c,d)=\fmul_{p/q,pq}(b,c,e)$, then $a\equiv b \pmod q$.\end{lemma}
\begin{proof}
\begin{flalign*}
&\fmul_{p/q,pq}(a,c,d)=\fmul_{p/q,pq}(b,c,e) \\
\implies &\mul_{p,pq}(\mul_{p,pq}(a,c),\mul_{p,pq}(c,d))=\mul_{p,pq}(\mul_{p,pq}(b,c),g_{p,pq}(c,e)) \\
\overset{L \ref{g1}}{\implies} &\mul_{p,pq}(a,c)\equiv \mul_{p,pq}(b,c)\pmod{q}
\overset{L \ref{g2}}{\implies} a\equiv b\pmod{q}.
\end{flalign*}
\end{proof}

\begin{proposition}\label{leftdet}There is a radius-$1$ CA $\Delta_{p/q}:\trsh(\Mul_{p/q,pq})\to\trsh(\Mul_{p/q,pq})$ such that $\Delta_{p/q}(\tr_{\Mul_{p/q,pq},i}(x))=\tr_{\Mul_{p/q,pq},i-1}(x)$ for all $x\in\digs_{pq}^\Z$, $i\in\Z$.\end{proposition}
\begin{proof}It suffices to restrict to the case $i=1$ and to show for an arbitrary $x\in\digs_{pq}^\Z$ that the value of $\tr_{\Mul_{p/q,pq},0}(x)[0]$ can be computed from $\tr_{\Mul_{p/q,pq},1}(x)[-1]$, $\tr_{\Mul_{p/q,pq},1}(x)[0]$ and $\tr_{\Mul_{p/q,pq},1}(x)[1]$ (by some function $\delta:\digs_{pq}^3\to\digs_{pq}$, which we will not explicitly derive). By the definition of the trace map this is equivalent to showing that $x[0]$ can be computed from $\Mul_{q/p,pq}(x)[1]$, $x[1]$ and $\Mul_{p/q,pq}(x)[1]$.

Because $\Mul_{p/q,pq}(x)[1]=\fmul_{p/q,pq}(x[0],x[1],x[2])$, by Lemma \ref{f1} the value of $x[0]$ modulo $q$ can be computed from $x[1]$ and $\Mul_{p/q,pq}(x)[1]$ (see Figure~\ref{ldetproof}, left). Similarly, because $\Mul_{q/p,pq}(x)[1]\allowbreak=\fmul_{q/p,pq}(x[0],x[1],x[2])$, by the same lemma the value of $x[0]$ modulo $p$ can be computed from $x[1]$ and $\Mul_{q/p,pq}(x)[1]$ (Figure~\ref{ldetproof}, middle). In total, the value of $x[0]$ both modulo $q$ and modulo $p$ can be computed from $\Mul_{q/p,pq}(x)[1]$, $x[1]$ and $\Mul_{p/q,pq}(x)[1]$ (Figure~\ref{ldetproof}, right). Because $x[0]\in \digs_{pq}$, this fully determines the value of $x[0]$.\end{proof}

\begin{figure}[h]
\centering
\begin{tikzpicture}
\def\tetris{ rectangle ++(30pt,30pt) ++(-60pt,0pt) rectangle ++(30pt,30pt) ++(0pt,-30pt) rectangle ++(30pt,30pt) ++(-30pt,0pt) rectangle ++(30pt,30pt)}

\node[anchor=east] at (-10pt,75pt) {$\Mul_{q/p,pq}(x)$:};
\node[anchor=east] at (-10pt,45pt) {$x$:};
\node[anchor=east] at (-10pt,15pt) {$\Mul_{p/q,pq}(x)$:};
\draw (30pt,0pt) \tetris;
\node[text width=20pt] at (15pt,45pt) {$1,3$ or $5$};
\node[minimum size=20pt] (l) at (15pt,45pt) {};
\node at (45pt,75pt) {?};
\node at (45pt,45pt) (l1) {$4$};
\node at (45pt,15pt) (l2) {$3$};
\draw[->,thick](l1.west) to (l.east); \draw[->,thick] (l2.west) to [bend left = 45]  (l.south);

\node at (75pt,45pt) {$\land$};

\draw (120pt,0pt) \tetris;
\node (m) at (105pt,45pt) {$0$ or $3$};
\node (m1) at (135pt,75pt) {$2$};
\node (m2) at (135pt,45pt) {$4$};
\node at (135pt,15pt) {?};
\draw[->,thick](m1.west) to [bend right = 45] (m.north); \draw[->,thick] (m2.west) to (m.east);

\node at (170pt,45pt) {$\implies$};

\draw (220pt,0pt) \tetris;
\node (r) at (205pt,45pt) {$3$};
\node (r1) at (235pt,75pt) {$2$};
\node (r2) at (235pt,45pt) {$4$};
\node (r3) at (235pt,15pt) {$3$};
\draw[->,thick](r1.west) to [bend right = 45] (r.north); \draw[->,thick] (r2.west) to (r.east); \draw[->,thick] (r3.west) to [bend left = 45]  (r.south);
\end{tikzpicture}
\caption{The proof of Proposition \ref{leftdet} (here $(p,n)=(3,6)$).}
\label{ldetproof}
\end{figure}
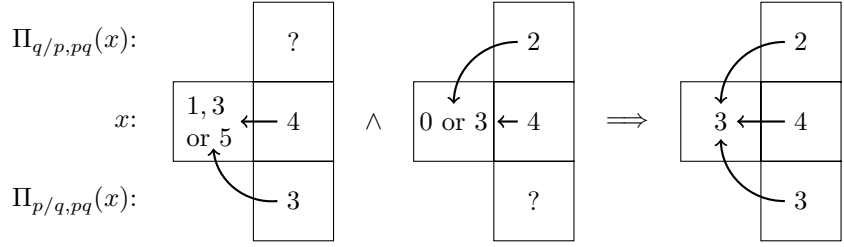

Proposition~\ref{leftdet} is not needed to prove our main results in Section~\ref{trSect}, but instead we use it to prove two simple results of independent interest. As the first application of this proposition we prove that $\Mul_{p/q,pq}$ cannot generate an eventually periodic trace from a configuration that represents a positive real number. The same idea has been used for other cellular automata in~\cite{Jen90}.

\begin{proposition}\label{mulAper}
Let $p>q$. If $x\in\digs_{pq}^\Z$ is a configuration that represents a positive real number (in particular, if $x$ is a finite configuration different from $0^\Z$), then $\tr_{\Mul_{p/q,pq}}(x)$ is not eventually periodic.
\end{proposition}
\begin{proof}
Let $x\in\digs_{pq}^\Z$ be such that $\real_{pq}(x)>0$. Assume to the contrary that $y=\tr_{\Mul_{p/q,pq}}(x)$ is eventually periodic, i.e. there are $P\in\Npos$, $i_0\in\N$ such that $y[i+P]=y[i]$ for all $i\geq i_0$, and we may assume that this holds even for all $i\in\N$ (by considering the configuration $\Mul_{p/q,pq}^{i_0}(x)$ instead of $x$ if necessary). Denote $x_t=(\sigma^{-1}\circ \Mul_{p/q,pq})^t(x)$ and $y_t=\tr_{\Mul_{p/q,pq}}(x_t)$ for all $t\in\N$. An inductive application of Proposition \ref{leftdet} with respect to $t$ shows that $y_t[i+P]=y_t[i]$ for all $i,t\in\N$.

Note that $\real_{pq}(x_t)=\left(\frac{1}{pq}\frac{p}{q}\right)^t\real_{pq}(x)=\real_{pq}(x)/q^{2t}$ for all $t\in\N$. Fix $T$ so that $\left(\frac{p}{q}\right)^P\real_{pq}(x_T)<1$. From this it follows that $y_T[-\infty,P]={}^\infty 0$ and by the eventual periodicity of $y_T$ it follows that $y_T=0^\Z$. Applying Proposition \ref{leftdet} shows that $y_t=0^\Z$ for all $t\geq T$. In particular $\Mul_{p/q,pq}^t(x)[-\infty,-T]={}^\infty 0$ for $t\in\N$ and the sequence $\left(\left(\frac{p}{q}\right)^t\real_{pq}(x)\right)_{t\in\N}$ is bounded from above by $(pq)^T$, which contradicts the assumption that $\real_{pq}(x)>0$.
\end{proof}

A second implication of Proposition~\ref{leftdet} is that to understand the dynamics of all trace subshifts of $\Mul_{p/q,pq}$ it is sufficient to study the trace subshifts of width $1$. This further justifies our focus on the trace shift $\trsh(\Mul_{p/q,pq})$. 

\begin{proposition}\label{univ1trace}
Let $i<j$ be integers and let $I=[i,j]\subseteq\Z$ be an interval of integers. Then the subshifts $\trsh_I(\Mul_{p/q,pq})$, $\trsh_j(\Mul_{p/q,pq})$ and $\trsh(\Mul_{p/q,pq})$ are conjugate.
\end{proposition}
\begin{proof}
We make the natural identification of $\trsh_I(\Mul_{p/q,pq})$ as a subset of the cartesian product $\trsh_i(\Mul_{p/q,pq})\times\cdots\times\trsh_j(\Mul_{p/q,pq})$. We define a map $F:\trsh_j(\Mul_{p/q,pq})\to \trsh_I(\Mul_{p/q,pq})$ by
\[F(x)=(\Delta_{p/q}^{j-i}(x),\Delta_{p/q}^{j-(i+1)}(x)\dots,\Delta_{p/q}(x),x)\]
for $x\in \trsh_j(\Mul_{p/q,pq})$. This is easily seen to be an injective sliding block code, and by Proposition~\ref{leftdet} it is also surjective, so $F$ is a conjugacy.

The subshifts $\trsh_j(\Mul_{p/q,pq})$ and $\trsh(\Mul_{p/q,pq})$ are identical as sets, so they are also conjugate.
\end{proof}

An important class of CA on full shifts are the \emph{permutive} cellular automata. We say that a CA $F:A^\Z\to A^\Z$ defined by a local rule $f:A^{d+1}\to A^{d+1}$ is left permutive if for every $w\in A^d$ it holds that $f(a,w)\neq f(b,w)$ whenever $a,b\in A$ are distinct (similarly one defines right permutive CA). This is equivalent to saying that the map $A\to A$ defined by $a\to f(a,w)$ is a permutation for every $w\in A^d$. The following lemma shows that $\fmul_{p/q,pq}$ has a kind of a \emph{partial} permutivity property: as the symbol $a$ varies modulo $q$, also $\fmul_{p/q,pq}(a,w)$ varies modulo $q$.

\begin{lemma}\label{f2}$\fmul_{p/q,pq}(a,c,d)\equiv \fmul_{p/q,pq}(b,c,d)\pmod q \\ \iff a\equiv b \pmod q\iff \fmul_{p/q,pq}(a,c,d)=\fmul_{p/q,pq}(b,c,d)$.\end{lemma}
\begin{proof}
\begin{flalign*}
&\fmul_{p/q,pq}(a,c,d)\equiv \fmul_{p/q,pq}(b,c,d)\pmod q \\
\iff &\mul_{p,pq}(\mul_{p,pq}(a,c),\mul_{p,pq}(c,d))\equiv \mul_{p,pq}(\mul_{p,pq}(b,c),\mul_{p,pq}(c,d))\pmod q \\
\overset{L \ref{g2}}{\iff} &\mul_{p,pq}(a,c)\equiv \mul_{p,pq}(b,c)\pmod{q}
\overset{L \ref{g2}}{\iff} a\equiv b\pmod{q} \\
\overset{L \ref{g2}}{\iff} &\mul_{p,pq}(\mul_{p,pq}(a,c),\mul_{p,pq}(c,d))=\mul_{p,pq}(\mul_{p,pq}(b,c),\mul_{p,pq}(c,d)) \\
\iff &\fmul_{p/q,pq}(a,c,d)=\fmul_{p/q,pq}(b,c,d)
\end{flalign*}
\end{proof}

\begin{corollary}\label{strongxz}$\fmul_{p/q,pq}(a,c,d)=\fmul_{p/q,pq}(b,c,e)\\ \implies\fmul_{p/q,pq}(a,c,d)=\fmul_{p/q,pq}(a,c,e)$.\end{corollary}
\begin{proof}By Lemma \ref{f1} $a\equiv b\pmod q$, so
\[\fmul_{p/q,pq}(a,c,d)=\fmul_{p/q,pq}(b,c,e)\overset{L\ref{f2}}{=}\fmul_{p/q,pq}(a,c,e).\]
\end{proof}

On the other hand, we show that as the symbol $a$ varies modulo $q$, the value of $\fmul_{p/q,pq}(a,w)$ remains constant modulo $p$. This is proved by reduction to $\mul_{p,pq}$. 

\begin{lemma}\label{g3}$\mul_{p,pq}(a,c)\equiv \mul_{p,pq}(b,c)\pmod p$.\end{lemma}
\begin{proof}Let $a=a_1q+a_0$, $b=b_1q+b_0$ and $c=c_1q+c_0$. Then
\[\mul_{p,pq}(a,c)=a_0p+c_1\equiv b_0p+c_1=\mul_{p,pq}(b,c)\pmod p.\]
\end{proof}

\begin{lemma}\label{f3}$\fmul_{p/q,pq}(a,c,d)\equiv \fmul_{p/q,pq}(b,c,d)\pmod p$.\end{lemma}
\begin{proof}
\begin{flalign*}
&\fmul_{p/q,pq}(a,c,d)=\mul_{p,pq}(\mul_{p,pq}(a,c),\mul_{p,pq}(c,d)) \\
\overset{L \ref{g3}}{\equiv}&\mul_{p,pq}(\mul_{p,pq}(b,c),\mul_{p,pq}(c,d))=\fmul_{p/q,pq}(b,c,d)\pmod p.
\end{flalign*}
\end{proof}

For any $a\in \digs_{pq}$ denote
\[Q_{p,q}(a)=\{d\in \digs_{pq}\mid d\equiv a \pmod p\}.\]
The set $Q_{p,q}(a)$ contains $q$ elements, all non-congruent modulo $q$. In particular $Q_{p,q}(a)$ is a complete residue system modulo $q$.

\begin{proposition}\label{flip}Let $Q\subseteq \digs_{pq}$ contain a complete residue system modulo $q$ and let $w\in \digs_{pq}^*$ be such that $\abs{w}\geq 2$. Then
\[\fmul_{p/q,pq}(Qw)=Q_{p,q}(b)w'\]
for some $b\in \digs_{pq}$ and $w'\in \digs_{pq}^*$, $\abs{w'}=\abs{w}-2$. In particular this holds when $Q=Q_{p,q}(a)$ for any $a\in \digs_{pq}$.
\end{proposition}
\begin{proof}
It is sufficient to prove this for words $w\in\digs_{pq}^2$ of length $2$. Let $a\in Q$ be arbitrary and $b=\fmul_{p/q,pq}(a,w[1],w[2])$. By Lemma \ref{f3} $\fmul_{p/q,pq}(Qw)\subseteq Q_{p,q}(b)$. To prove equality it is sufficient to show that $\abs{\fmul_{p/q,pq}(Qw)}=q$, but this follows from Lemma \ref{f2}.
\end{proof}

Consider two configurations that represent the same number in base $6$, e.g. $\cdots000.300\cdots$ and $\cdots000.255\cdots$ that represent the number $1/2$. From the facts that $\Mul_{3/2,6}$ is bijective and maps finite configurations to finite configurations it follows that these two configurations are mapped to $\cdots000.4300\cdots$ and $\cdots000.4255\cdots$ respectively, i.e. to the two base-$6$ representatives of the number $3/4$. In this case one can also observe that the infinite sequences $300\cdots$ and $255\cdots$ are shifted by one position to the right by the action of $\Mul_{3/2,6}$. This observation is generalized in the following lemma and its corollary.

\begin{lemma}\label{tailInt}Let $Q=\{np\mid 1\leq n<q\}\subseteq \digs_{pq}$. For any $s\in Q$, $j\in\Z$ define $e_{s,j},e_{s-1,j}\in \digs_{pq}^\Z$ by
\begin{flalign*}
\begin{array}{l l}
e_{s,j}[i]=
\left\{
\begin{array}{l}
s \text{ when } i=j,\\
0 \text{ when } i>j,
\end{array}
\right.
&
e_{s-1,j}[i]=
\left\{
\begin{array}{l}
s-1 \text{ when } i=j,\\
pq-1 \text{ when } i>j
\end{array}
\right.
\end{array}
\end{flalign*}
(their values at $i<j$ are irrelevant). For any $x\in \digs_{pq}^\Z$, $s\in Q$ and $j\in\Z$ there exist $x'\in \digs_{pq}^\Z$ and $s'\in Q$ such that
\[\Mul_{p,pq}(x\glue_j e_{s,j})=x'\glue_j e_{s',j} \hspace{0.3cm} \text{and} \hspace{0.3cm} \Mul_{p,pq}(x\glue_j e_{s-1,j})=x'\glue_j e_{s'-1,j}.\]
\end{lemma}
\begin{proof}
Denote $x_1=\Mul_{p,pq}(x\glue_j e_{s,j})$ and $x_2=\Mul_{p,pq}(x\glue_j e_{s-1,j})$. Clearly $x_1[i]=x_2[i]$ for $i\leq j-2$. The claim that $x_1[i]=0$ and $x_2[i]=pq-1$ for $i>j$ follows by checking that $\mul_{p,pq}(0,0)=0$ and $\mul_{p,pq}(pq-1,pq-1)=pq-1$. It remains to show that $x_1[j-1]=x_2[j-1]$, $x_1[j]=s'$ and $x_2[j]=s'-1$ for some $s'\in Q$. 

Let us write $x[j-1]=a_1q+a_0$, $s=s_1q+s_0$ and $s-1=s_1q+(s_0-1)$ where $a_1,s_1\in \digs_p$ and $a_0,s_0,s_0-1\in \digs_q$: this is possible because $s$ is not divisible by $q$. Then
\begin{flalign*}
x_1[j-1]&=\mul_{p,pq}(x[j-1],s)=a_0p+s_1=\mul_{p,pq}(x[j-1],s-1)=x_2[j-1], \\
x_1[j]&=\mul_{p,pq}(s,0)=\mul_{p,pq}(s_1q+s_0,0q+0)=s_0p\doteqdot s'\in Q, \\
x_2[j]&=\mul_{p,pq}(s-1,pq-1)=\mul_{p,pq}(s_1q+(s_0-1),(p-1)q+(q-1)) \\
&=(s_0-1)p+(p-1)=s'-1.
\end{flalign*}
\end{proof}

\begin{corollary}\label{tail}Using the notation of the previous lemma, for any $x\in \digs_{pq}^Z$, $s\in Q$ and $j\in\Z$ there exist $x'\in \digs_{pq}^\Z$ and $s'\in Q$ such that
\[\Mul_{p/q,pq}(x\glue_j e_{s,j})=x'\glue_{j+1} e_{s',j+1} \hspace{0.3cm} \text{and} \hspace{0.3cm} \Mul_{p/q,pq}(x\glue_j e_{s-1,j})=x'\glue_j e_{s'-1,j+1}.\]
\end{corollary}

\section{Traces of Fractional Multiplication Automata}\label{trSect}

In this section we assume that $p>q>1$ are coprime integers unless otherwise specified. We will prove our main results: the trace subshift $\trsh(\Mul_{p/q,pq})$ is not synchronizing and its intersection with $\digs_p^\Z$ is not sofic.

To simplify the notation, we will denote for coprime $s,t>1$ (not necessarily $s>t$) $\tr_{s/t,I}(x)=\tr_{\Mul_{s/t,st},I}(x)$, $\trsh_{s/t}=\trsh(\Mul_{s/t,st})$, $L(s/t)=L(\trsh_{s/t})$ and $\pre_{s/t}=\pre_{\trsh_{s/t}}$. We will abuse notation and define the trace with respect to $\Mul_{s/t,st}$ also for positive real numbers.

\begin{definition}For $\xi\in\Rpos$ we call sequence
\[\tr_{s/t}(\xi)=\tr_{s/t}(\config_{st}(\xi))\]
the \emph{trace} $s/t$\emph{-representation} of $\xi$.
\end{definition}

Since $\config_{st}(\Rpos)$ is a dense subset of $\digs_{st}^\Z$, it follows that $\trsh_{s/t}$ is the topological closure of $\tr_{s/t}(\Rpos)$.

Following \cite{AFS08}, let $\psi_{p/q}:\Rpos\to\Z$ be the function defined by
\[\psi_{p/q}(\xi)=q\left\lfloor\frac{p}{q}\xi\right\rfloor-p\lfloor \xi\rfloor=p\fractional(\xi)-q\fractional\left(\frac{p}{q}\xi\right).\]
This function is periodic of period $q$ and for every $\xi\in\Rpos$, $\psi_{p/q}(\xi)$ belongs to the set
\[\digs_{-q,p}\doteqdot\{-(q-1),\dots,0,1,\dots (p-1)\}.\]

\begin{definition}For every $\xi\in\Rpos$, the infinite sequence $\varphi_{p/q}(\xi)\subseteq \digs_{-q,p}^\Z$ defined by
\[\varphi_{p/q}(\xi)[i]=\psi_{p/q}\left(\left(\frac{p}{q}\right)^i \xi\right)\text{ for every } i\in\Z\]
is called the \emph{companion} $p/q$\emph{-representation} of $\xi$. The topological closure of $\varphi_{p/q}(\Rpos)\subseteq \digs_{-q,p}^\Z$
is a subshift denoted by $Y_{p/q}$.\end{definition}

The subscript $p/q$ is omitted from all notations when it is clear from the context.

The name ``companion $p/q$-representation'' was introduced in \cite{AFS08}, probably to signify its connection to another type of a number representation system considered in the same paper. We adopt the same name because it will turn out that the companion $p/q$-representations are also strongly connected to trace $p/q$-representations. The earliest occurrence of the sequence $\varphi(\xi)$ seems to be in a paper of Forman and Shapiro \cite{FS67} (where it has not been named). This representation, and its generalizations, also comes up in a sequence of papers by Dubickas starting from \cite{Dub05}.

The following lemma from~\cite{AFS08} shows that $\varphi(\xi)$ really is in some sense a representation of $\xi$ in base $p/q$. 

\begin{lemma}\label{compfrac}$\fractional(\xi)=\frac{1}{p}\sum_{i=0}^{\infty}\left(\frac{q}{p}\right)^i\varphi(\xi)[i]$ for every $\xi\in\Rpos$.\end{lemma}
\begin{proof}For $i\in\N$ denote $y_i=\fractional((p/q)^i \xi)$ and $s_i=\varphi(\xi)[i]=py_i-qy_{i+1}$. From this we can solve
\[y_0=\frac{1}{p}s_0+\frac{q}{p}y_1=\frac{1}{p}s_0+\frac{1}{p}\frac{q}{p}s_1+\left(\frac{q}{p}\right)^2 y_2=\cdots=\frac{1}{p}\sum_{i=0}^{\infty}\left(\frac{q}{p}\right)^i s_i.\]
\end{proof}

\begin{definition}For $n>1$ define $\md_n:\Z\to \digs_n$ by 
\[\md_n(m)=m-n\lfloor m/n\rfloor,\] i.e. $\md_n(m)$ is the remainder of $m$ divided by $n$. It can be extended to a function $\Z^\Z\to \digs_n^\Z$ by coordinatewise application.
\end{definition}

\begin{definition}\label{phiDef}For every $x\in \digs_{pq}^\Z$ define the bi-infinite sequence $\Phi(x)$ by
\[\Phi(x)[i]=q\md_p(x[i+1])-p\md_q(x[i])\text{ for every } i\in\Z.\]
\end{definition}

The map $\Phi$ connects the two different $p/q$ representations. 

\begin{theorem}\label{repchange}$\Phi(\tr(\xi))=\varphi(\xi)$ for every $\xi\in\Rpos$.\end{theorem}
\begin{proof}For every $i\in\Z$ we can write
\[\left(\frac{p}{q}\right)^i \xi=n_i q+a_i+\xi_i,\]
where $n_i\in\N$, $a_i\in \digs_q$ and $\xi_i\in[0,1)$ are unique. Then
\[\left(\frac{p}{q}\right)^{i+1} \xi= n_i p+\frac{p}{q}(a_i+\xi_i)=n_i p+ b_i+\xi_i'\]
for unique $b_i\in \digs_p$ and $\xi_i'\in[0,1)$, because $\frac{p}{q}(a_i+\xi_i)\in [0,p)$. Thus
\begin{flalign*}
\varphi(\xi)[i]&=\psi_{p/q}\left(\left(\frac{p}{q}\right)^i \xi\right)=q\left\lfloor\left(\frac{p}{q}\right)^{i+1}\xi\right\rfloor-p\left\lfloor\left(\frac{p}{q}\right)^i \xi\right\rfloor \\
&=q\lfloor n_i p+ b_i+\xi_i' \rfloor-p\lfloor n_i q+a_i+\xi_i\rfloor=qb_i-pa_i \\
&=q\md_p(\tr(\xi)[i+1])-p\md_q(\tr(\xi)[i])=\Phi(\tr(\xi))[i].
\end{flalign*}
\end{proof}

The correspondence between the two $p/q$ representations extends to the level of the induced subshifts $\trsh_{p/q}$ and $Y_{p/q}$.

\begin{theorem}$\Phi:\trsh_{p/q}\to Y_{p/q}$ is a conjugacy.\end{theorem}
\begin{proof}We first prove that $\Phi$ is injective on $\digs_{pq}^\Z$. To see this, assume that $x,y\in \digs_{pq}^\Z$ are elements such that $\Phi(x)=\Phi(y)$ and let $i\in\Z$. Then
\begin{flalign*}
&\Phi(x)[i-1]=\Phi(y)[i-1] \\
\implies& q\md_p(x[i])-p\md_q(x[i-1])=q\md_p(y[i])-p\md_q(y[i-1]) \\
\implies& q\md_p(x[i])\equiv q\md_p(y[i]) \pmod p \implies x[i]\equiv y[i] \pmod p
\end{flalign*}
and
\begin{flalign*}
&\Phi(x)[i]=\Phi(y)[i] \\
\implies& q\md_p(x[i+1])-p\md_ q(x[i])=q\md_p(y[i+1])-p\md_q(y[i]) \\
\implies& p\md_q(x[i])\equiv p\md_q(y[i]) \pmod q \implies x[i]\equiv y[i] \pmod q.
\end{flalign*}
Because $x[i],y[i]\in \digs_{pq}$, it follows that $x[i]=y[i]$ for all $i\in\Z$.

Since $\Phi:\digs_{pq}^\Z\to\Phi(\digs_{pq}^\Z)$ is a continuous injective map on a compact metrizable space, it is a homeomorphism. Using the previous theorem we can deduce that
\[\Phi(\trsh_{p/q})=\Phi(\overline{\tr_{p/q}(\Rpos)})=\overline{\Phi(\tr_{p/q}(\Rpos))}=\overline{\varphi(\Rpos)}=Y_{p/q}:\]
because $\Phi$ is a homeomorphism, we can change the order of taking a topological closure and applying $\Phi$. Therefore the restriction map $\Phi:\trsh_{p/q}\to Y_{p/q}$ is continuous and bijective.

Consider now $\Phi$ restricted to $\trsh_{p/q}$. To conclude, we need to show that $\Phi\circ\sigma_{\trsh_{p/q}}=\sigma_{Y_{p/q}}\circ\Phi$. But this follows directly from Definition~\ref{phiDef}, which gives $\Phi$ as a sliding block code.
\end{proof}

A special case of Lemma 1 in~\cite{Dub05} says that $\varphi_{p/q}(\xi)$ is not eventually periodic for $\xi\in\Rpos$. The last two theorems together with Proposition~\ref{mulAper} yield an alternative proof of this fact.

We begin to examine the properties of the language $L(p/q)$ with the aim of proving that $\trsh_{p/q}$ is not sofic or synchronizing. 

\begin{lemma}\label{restr}If $a_1,a_2,b_1,b_2\in \digs_{pq}$, $w\in \digs_{pq}^n$ for some $n$, $a_1\not\equiv a_2\pmod q$ and $b_1\not\equiv b_2\pmod p$, then $\{a_i w b_j\mid i,j\in\{1,2\}\}\not\subseteq L(p/q)$.\end{lemma}
\begin{proof}Assume to the contrary that $\{a_i w b_j\mid i,j\in\{1,2\}\}\subseteq L(p/q)$. Without loss of generality $m_q=\md_q(a_2)-\md_q(a_1)>0$ and $m_p=\md_p(b_1)-\md_p(b_2)>0$. Let $\xi_1,\xi_2\in\Rpos$ be such that $\tr(\xi_i)[0,n+1]=a_i w b_i$ for $i\in\{1,2\}$. For any $\xi\in\Rpos$ we have
\begin{flalign*}
&\frac{1}{q^{n+1}}\psi_{p^{n+1}/q^{n+1}}(\xi)=\left\lfloor\left(\frac{p}{q}\right)^{n+1} \xi\right\rfloor-\left(\frac{p}{q}\right)^{n+1}\lfloor \xi\rfloor \\
&=\sum_{i=0}^{n}\left(\frac{p}{q}\right)^i\left(\left\lfloor\left(\frac{p}{q}\right)^{n-i+1} \xi\right\rfloor -\left(\frac{p}{q}\right)\left\lfloor\left(\frac{p}{q}\right)^{n-i} \xi\right\rfloor\right) \\
&=\frac{1}{q}\sum_{i=0}^{n}\left(\frac{p}{q}\right)^i\varphi_{p/q}(\xi)[n-i] \\
&\overset{T \ref{repchange}}{=}\frac{1}{q}\sum_{i=0}^{n}\left(\frac{p}{q}\right)^i(q\md_p(\tr(\xi)[n-i+1])-p\md_q(\tr(\xi)[n-i])),
\end{flalign*}
and because $\tr(\xi_1)[1,n]=\tr(\xi_2)[1,n]$, it follows that
\begin{flalign*}
&\psi_{p^{n+1}/q^{n+1}}(\xi_1)-\psi_{p^{n+1}/q^{n+1}}(\xi_2) \\
&=q^n\left(q\md_p(\tr(\xi_1)[n+1])-\left(\frac{p}{q}\right)^n p\md_q(\tr(\xi_1)[0])\right) \\
&-q^n\left(q\md_p(\tr(\xi_2)[n+1])-\left(\frac{p}{q}\right)^n p\md_q(\tr(\xi_2)[0])\right) \\
&=q^n\left(qm_p+\left(\frac{p}{q}\right)^n pm_q\right)=q^{n+1}m_p+p^{n+1}m_q\geq q^{n+1}+p^{n+1}
\end{flalign*}
which contradicts the fact that $\psi_{p^{n+1}/q^{n+1}}(\xi)\in \digs_{-q^{n+1},p^{n+1}}$ for all $\xi\in\Rpos$.
\end{proof}

\begin{lemma}\label{flipword}Let $s,t>1$ be coprime (we do not assume that $s>t$). If $wa\in L(s/t)$ for some $w\in \digs_{st}^+$ and $a\in \digs_{st}$, then $wQ_{s,t}(a)\subseteq L(s/t)$.\end{lemma}
\begin{proof}Let $x\in \digs_{st}^\Z$ such that $\tr_{s/t}(x)[0,\abs{wa}-1]=wa$ and for every $d\in \digs_{st}$ let $x_d\in \digs_{st}^\Z$ be such that $x_d(-\abs{w})=d$ and $x_d[i]=x[i]$ for $i\neq -\abs{w}$. Then
\[\{\tr_{s/t}(x_d)[0,\abs{wa}-1]\mid d\in \digs_{st}^\Z\}=wQ_{s,t}(a)\]
by repeated application of Proposition~\ref{flip}.
\end{proof}

\begin{lemma}For any $a\in \digs_{pq}$ it holds that
\begin{flalign*}
\abs{\fmul_{q/p,pq}(0,a,\digs_{pq})}=
\left\{
\begin{array}{l}
2 \text{ when } \md_p(aq)\in\{p-i\mid 1\leq i\leq q-1\},\\
1 \text{ otherwise.}
\end{array}
\right.
\end{flalign*}
Moreover, in the first case, there are $d_a,b_{a,1}, b_{a,2}=b_{a,1}+1\in \digs_{pq}$ such that $b_{a,2}$ is divisible by $p$ and $\fmul_{q/p,pq}(0,a,b_{a,1})=d_a$, $\fmul_{q/p,pq}(0,a,b_{a,2})=d_a+1$.
\end{lemma}
\begin{proof}
For $a,b\in \digs_{pq}$ write $a=a_1p+a_0$ and $b=b_1p+b_0$. Then
\[\fmul_{q/p,pq}(0,a,b)=\mul_{q,pq}(\mul_{q,pq}(0,a),\mul_{q,pq}(a,b))=\mul_{q,pq}(a_1,a_0 q+b_1)=\md_p(a_1)q+d,\]
where $d=\left\lfloor\frac{a_0q+b_1}{p}\right\rfloor$. If $b$ ranges over $\digs_{pq}$, then $b_1$ ranges over $\digs_q$ and $d$ can attain two distinct values if and only if $\md_p(aq)=\md_p(a_0q)\in\{p-i\mid 1\leq i\leq q-1\}$.

If $d$ can attain two distinct values, then there is a unique $c\in \digs_q\setminus\{q-1\}$ such that $\left\lfloor\frac{a_0q+c}{p}\right\rfloor<\left\lfloor\frac{a_0q+(c+1)}{p}\right\rfloor$. Then we can choose $b_{a,1}=cp+(p-1)$ and $b_{a,2}=(c+1)p$.
\end{proof}

\begin{lemma}\label{contexts}For any $a\in \digs_{pq}$ there is a $d\in\digs_{pq}$ such that
\begin{flalign*}
\pre_{p/q}(a)=
\left\{
\begin{array}{l}
Q_{q,p}(d)\cup Q_{q,p}(d+1) \text{ if } \md_p(aq)\in\{p-i\mid 1\leq i\leq q-1\},\\
Q_{q,p}(d) \text{ otherwise.}
\end{array}
\right.
\end{flalign*}
In particular, $\abs{\pre_{p/q}(a)}$ is equal to $2p$ or $p$ respectively.
\end{lemma}
\begin{proof}
If $\md_p(aq)\in\{p-i\mid 1\leq i\leq q-1\}$, then by the previous lemma there is a partition $B_1\cup B_2=\digs_{pq}$ such that $\fmul_{q/p,pq}(0,a,B_1)=d$ and $\fmul_{q/p,pq}(0,a,B_2)=d+1$ for some $d\in \digs_{pq}$. Then by applying Proposition \ref{flip} it follows that $\fmul_{q/p,pq}(\digs_{pq},a,B_1)=Q_{q,p}(d)$ and $\fmul_{q/p,pq}(\digs_{pq},a,B_2)=Q_{q,p}(d+1)$, so we have $\pre_{p/q}(a)=Q_{q,p}(d)\cup Q_{q,p}(d+1)$. This is a set of cardinality $2p$. The proof for $\md_p(aq)\notin\{p-i\mid 1\leq i\leq q-1\}$ is similar.
\end{proof}

\begin{lemma}\label{wordcontexts}For any $w\in L(p/q)\setminus\{\epsilon\}$ there is a $d\in\digs_{pq}$ such that either $\pre_{p/q}(w)=Q_{q,p}(d)\cup Q_{q,p}(d+1)$ or $\pre_{p/q}(w)=Q_{q,p}(d)$. In particular, $\abs{\pre_{p/q}(w)}$ is equal to $2p$ or $p$.
\end{lemma}
\begin{proof}
Consider an arbitrary word $w=av\in L(p/q)$, where $v\in \digs_{pq}^*$ and $a\in \digs_{pq}$. Evidently $\pre_{p/q}(av)\neq\emptyset$ and by the previous lemma $\pre_{p/q}(av)\subseteq\pre_{p/q}(a)\subseteq Q_{q,p}(d)\cup Q_{q,p}(d+1)$ for some $d\in\digs_{pq}$. Then from Lemma \ref{flipword} it follows that $\pre_{p/q}(av)=\bigcup_{i\in\ind{I}}Q_{q,p}(d+i)$ for some nonempty set $\ind{I}\subseteq\{0,1\}$. \end{proof}

Based on this lemma we define two sets of words for every $n\in\Npos$:

\begin{flalign*}
&W_{1,n}=\left\{w\in L(p/q)\cap \digs_{pq}^n\mid \abs{\pre_{p/q}(w)}=p\right\} \\
&W_{2,n}=\left\{w\in L(p/q)\cap \digs_{pq}^n\mid \abs{\pre_{p/q}(w)}=2p\right\}.
\end{flalign*}
These form a partition $L(p/q)\cap \digs_{pq}^n=W_{1,n}\cup W_{2,n}$. In the next two lemmas we show how to find all elements of $W_{2,n}$ in the traces of suitable configurations.

\begin{lemma}Let $s\in Q$ and $e_{s,0}$ be as in Lemma \ref{tailInt}. Then we have $\tr_{p/q}(x\glue_0 e_{s,0})[1,n]\in W_{2,n}$ for every $x\in \digs_{pq}^\Z$ and $n\in\Npos$.\end{lemma}
\begin{proof}Let $w=\tr_{p/q}(x\glue_0 e_{s,0})[1,n]$. By Corollary \ref{tail}
\[\tr_{p/q}(x\glue_0 e_{s,0})[1,n]=\tr_{p/q}(x\glue_0 e_{s-1,0})[1,n],\]
so we have $sw,(s-1)w\in L(p/q)$. By Lemma \ref{flipword} $\pre_{p/q}(w)$ contains at least $2p$ words, so $w\in W_{2,n}$.\end{proof}

\begin{lemma}\label{construct}Let $Q=\{np\mid 1\leq n<q\}$ and fix $n\in\Npos$. For every $s\in Q$ the set
\[W_s=\{\tr_{p/q}(x)[1,n]\mid x\in \digs_{pq}^\Z, x[0]=s, x[i]=0\text{ for } i>0\}\subseteq W_{2,n}\]
contains $q^n$ elements, $W_{2,n}=\bigcup_{s\in Q}W_s$ and $\abs{W_{2,n}}=q^n(q-1)$.\end{lemma}
\begin{proof}Denote $W=\bigcup_{s\in Q}W_s$. We begin by showing that $W\subseteq W_{2,n}$ and that $\abs{W}=q^n(q-1)$. First, $W_s\subseteq W_{2,n}$ follows from the previous lemma, and by repeated application of Proposition \ref{flip} it follows that $\abs{W_s}=q^n$. To prove that $\abs{W}=q^n(q-1)$ it is enough to show that $W_s\cap W_{s'}=\emptyset$ for distinct $s,s'\in Q$. This in turn follows by showing that $\fmul_{p/q,pq}(a,s,0)\neq \fmul_{p/q,pq}(b,s',0)$ for all $a,b\in \digs_{pq}$. Therefore let $a=a_1q+a_0$, $b=b_1q+b_0$, $s=s_1q+s_0$ and $s'=s'_1q+s'_0$. Let $d_1,d'_1\in \digs_p$ and $d_0, d'_0\in \digs_q$ be such that $s_0p=d_1q+d_0$ and $s'_0p=d'_1q+d'_0$. Since $s,s'\in Q$, we have $s\not\equiv s'\pmod q$ so the values $s_0,s'_0\in \digs_q$ are distinct. Then $\abs{s_0p-s'_0p}\geq p>q$, so $d_1\neq d'_1$. We compute
\begin{flalign*}
&\fmul_{p/q,pq}(a,s,0)=\mul_{p,pq}(\mul_{p,pq}(a,s),\mul_{p,pq}(s,0))=\mul_{p,pq}(a_0p+s_1,s_0p) \\
&=\md_q(a_0p+s_1)p+d_1\not\equiv\md_q(b_0p+s'_1)p+d'_1=\fmul_{p/q,pq}(b,s',0)\pmod p .
\end{flalign*}

To prove the inclusion $W_{2,n}\subseteq W$ it is now sufficient to show that $\abs{W_{2,n}}=q^n(q-1)$. The proof is by induction. The case $n=1$ follows from Lemma \ref{contexts}, so let us assume that the claim holds for some $n\in\Npos$. By the previous paragraph $\abs{W_{2,n+1}}\geq q^{n+1}(q-1)$, so let us assume contrary to our claim that $\abs{W_{2,n+1}}>q^{n+1}(q-1)$. Every element of $W_{2,n+1}$ is of the form $wa$ where $w\in W_{2,n}$ and $a\in \digs_{pq}$, so by pigeonhole principle there exist $w\in W_{2,n}$ and letters $a_1,a_2,\dots,a_k\in \digs_{pq}$ with $k>q$ such that $wa_i\in W_{2,n+1}$ for all $1\leq i\leq k$. Without loss of generality $a_1\not\equiv a_2\pmod p$ and $\pre_{p/q}(wa_1)=\pre_{p/q}(w)=\pre_{p/q}(wa_2)$, which contradicts Lemma \ref{restr}.\end{proof}

This characterization of the set $W_{2,n}$ will be of use in proving that $\trsh_{p/q}$ is not sofic or synchronizing. As a byproduct we found the cardinality of $W_{2,n}$, which allows us to compute the complexity function of $\trsh_{p/q}$.

\begin{theorem}\label{combi}$P_{\trsh_{p/q}}(n)=pq(p^{n-1}-q^{n-1})\frac{q-1}{p-q}+p^n q$ for every $n\in\Npos$.\end{theorem}
\begin{proof}The proof is by induction. In the case $n=1$ the expression equals $pq$, so let us assume that the equation holds for some $n\in\Npos$. Then
\begin{flalign*}
&P_{\trsh_{p/q}}(n+1)=2p\abs{W_{2,n}}+p\abs{W_{1,n}}=2p\abs{W_{2,n}}+p(P_{\trsh_{p/q}}(n)-\abs{W_{2,n}}) \\
&=p(\abs{W_{2,n}}+P_{\trsh_{p/q}}(n))=p\left(q^n(q-1)+pq(p^{n-1}-q^{n-1})\frac{q-1}{p-q}+p^n q\right) \\
&=pq^n(q-1)\frac{p-q}{p-q}+p^2 q(p^{n-1}-q^{n-1})\frac{q-1}{p-q}+p^{n+1}q \\
&=\left(pq^n(p-q)+p^2 q(p^{n-1}-q^{n-1})\right)\frac{q-1}{p-q}+p^{n+1}q \\
&=\left(p^2 q^n-pq^{n+1}+p^{n+1}q-p^2 q^n\right)\frac{q-1}{p-q}+p^{n+1}q \\
&=pq(p^n-q^n)\frac{q-1}{p-q}+p^{n+1}q.
\end{flalign*}
\end{proof}

\begin{example}For $p/q=3/2$, this is $6(3^{n-1}-2^{n-1})+3^n\cdot 2=4\cdot 3^n-3\cdot 2^n$. The first few terms are $6,24,84,276,876,\dots$\end{example}

\begin{lemma}\label{noperiodtrace}Let $Q=\{np\mid 1\leq n<q\}$, $j\in\Z$ and $x\in \digs_{pq}^\Z$ such that $x[j]\in Q$ and $x[i]=0$ for $i>j$. Then $\tr_{p/q}(x)$ is not eventually periodic.\end{lemma}
\begin{proof}Assume to the contrary that there are $i_0\in\N$, $P\in\Npos$ such that $\tr_{p/q}(x)[i]=\tr_{p/q}(x)[i+P]$ for $i\geq i_0$. By Lemma \ref{tail} we can see that $\Mul_{p/q,pq}^{M}(x)[j+M]\in Q$ and $\Mul_{p/q,pq}^{M}(x)[i+M]=0$ for any $M\in\N$ and for $i>j$, so without loss of generality (by considering the configuration $\Mul_{p/q,pq}^M(x)$ instead of $x$ for sufficiently large $M$ if necessary) $i_0=0$ and $j\geq 0$.

For each $n\in\Npos$ let $x_n=0^\Z\glue_{-((n+1)P+1)}x$, so for every $0\leq i\leq nP+1$ it holds that
\[\tr_{p/q}(x_n)[i]=\tr_{p/q}(x)[i]=\tr_{p/q}(x)[i+P]=\tr_{p/q}(x_n)[i+P]. \]
By Theorem \ref{repchange} $\Phi(\tr(x_n))=\varphi(\real(x_n))$, so it follows that
\[\varphi(\real(x_n))[i]=\varphi(\real(x_n))[i+P]=\varphi\left(\left(\frac{p}{q}\right)^P \real(x_n)\right)[i]\text{ for } 0\leq i\leq nP,\]
which by Lemma \ref{compfrac} implies that $\abs{\fractional(\real(x_n))-\fractional((p/q)^P\real(x_n))}=\mathcal{O}((q/p)^{nP})$. On the other hand, by Lemma \ref{tail}
\[x_n[j]\in Q,\quad \Mul_{p/q,pq}^P(x_n)[j+P]\in Q,\quad x_n[i]=\Mul_{p/q,pq}^P(x_n)[i+P]=0 \mbox{ for } i>j.\]
Since $x_n$ and $\Mul_{p/q,pq}^P(x_n)$ are base-$pq$ representations of the numbers $\real(x_n)$ and $(p/q)^P\real(x_n)$, it follows that
\[\abs{\fractional(\real(x_n))-\fractional((p/q)^P\real(x_n))}\geq (pq)^{-(j+P)},\]
a contradiction for sufficiently big $n\in\Npos$.
\end{proof}

\begin{theorem}\label{notsoficMahler}The subshift $\trsh_{p/q}\cap\digs_p^\Z$ is not sofic.\end{theorem}
\begin{proof}Assume to the contrary that $\trsh_{p/q}\cap\digs_p^\Z$ is sofic. We define $z\in \digs_{pq}^\Z$ as follows. First let $z[0]=p$ and $z[i]=0$ for $i\in\Npos$. Now let $i\in\Npos$ and assume that $z[0],\dots,z[-(i-1)]$ have been defined. By the permutivity property of Proposition \ref{flip} we can define $z[-i]$ in such a way that $\tr_{p/q}(z)[i]\in\digs_{p}$. By Lemma \ref{construct} the inclusion $\tr_{p/q}(z)[1,n]\in W_{2,n}$ holds for all $n\in\Npos$. A compactness argument together with Lemma \ref{wordcontexts} shows that $\pre_{p/q}(\tr_{p/q}(z)[1,\infty])=Q_{q,p}(d)\cup Q_{q,p}(d+1)$ for some $d\in\digs_{pq}$. We can choose $a\in Q_{q,p}(d)\cap\digs_p$ and $b\in Q_{q,p}(d+1)\cap\digs_p$ so in particular $a\not\equiv b\pmod q$.

We define $x_1,x_2\in \trsh_{p/q}$ as follows. First let $x_1[0,\infty]=a\tr_{p/q}(z)[1,\infty]$ and $x_2[0,\infty]=b\tr_{p/q}(z)[1,\infty]$. Now let $i\in\Npos$ and assume inductively that $x_1[0],\dots,x_1[-(i-1)]$ have been defined so that all prefixes of $x[-(i-1),\infty]$ are in $L(\trsh_{p/q})\cap\digs_p^*$. By Lemma \ref{flipword} and by compactness there exists $e\in\digs_{pq}$ such that $Q_{q,p}(e)\subseteq\pre_{p/q}(x[-(i-1),\infty])$. Then choose arbitrarily $x[i]\in Q_{q,p}(e)\cap\digs_{p}$. All the subwords of $x_1$ belong to $L(\trsh_{p/q})\cap\digs_p^*$ and therefore $x_1\in \trsh_{p/q}\cap\digs_p^\Z$. By the same argument we define $x_2$ so that $x_2\in \trsh_{p/q}\cap\digs_p^\Z$.

Define $x\in (\trsh_{p/q}\cap \digs_p^\Z)\times (\trsh_{p/q}\cap \digs_p^\Z)\subseteq(\digs_{p}^2)^\Z$ by $x[i]=(x_1[i],x_2[i])$ for $i\in\Z$. Since $(\trsh_{p/q}\cap \digs_p^\Z)\times (\trsh_{p/q}\cap \digs_p^\Z)$ is also sofic, by the pumping lemma of regular languages there exist $N,P\in\Npos$ such that

\[y_i=x_i[-\infty,N-1]x_i[N,N+P-1]^\infty\in \trsh_{p/q}\text{ for }i\in\{1,2\}.\]
Because $y_1[0]=x_1[0]=a\not\equiv b=x_2[0]=y_2[0]\pmod q$ and $y_1[i]=y_2[i]$ for $i>0$, it follows that $y_1[1,n]\in W_{2,p}$ for every $n\in\Npos$, so by compactness and by Lemma \ref{construct} there exists $y\in \digs_{pq}^\Z$ such that $y[0]\in Q=\{np\mid 1\leq n<q\}$, $y[i]=0$ for $i>0$ and $\tr_{p/q}(y)[1,\infty]=y_1[1,\infty]$: in particular $\tr_{p/q}(y)[i]=\tr_{p/q}(y)[i+P]$ for every $i\geq N$, which contradicts the previous lemma.
\end{proof}

\begin{corollary}\label{notreg}The subshift $\trsh_{p/q}$ is not sofic. In particular, the CA $\Mul_{p/q,pq}$ is not regular.\end{corollary}
\begin{proof}Assume to the contrary that $\trsh_{p/q}$ is sofic. Then $\trsh_{p/q}\cap\digs_p^\Z$ is also sofic as the intersection of two sofic subshifts, but this is impossible by the previous theorem.

The subshift $\trsh_{p/q}$ is an example of a non-sofic subshift factor of $\Mul_{p/q,pq}$, so it cannot be regular.
\end{proof}

We mention in passing that Jalonen and Kari show in Proposition~6 of~\cite{JK18} that there exists a reversible CA on a full shift which is left expansive (stated in~\cite{JK18} for right expansive CA) and has a non-sofic trace subshift. The previous corollary gives an alternative proof of this fact, because it follows from Proposition \ref{leftdet} that $\Mul_{p/q,pq}$ is left expansive.

The fact that $\trsh_{p/q}$ is not sofic was presented as a corollary of Theorem~\ref{notsoficMahler}. We proved the theorem on the subshift $\trsh_{p/q}\cap\digs_p^\Z$ due to the its connection to Mahler's $3/2$-problem. If we had no interest in the restriction of the subshift $\trsh_{p/q}$ to the alphabet $\digs_p$, then it would have been easier to prove Corollary~\ref{notreg} directly.

Our result has room for refinement. To ask further questions, let us generalize the language theoretical classification of CA. Instead of the class of sofic subshifts, one can consider an arbitrary class $\mathcal{C}$ of subshifts and ask whether all subshift factors of a CA $F$ belong to the class $\mathcal{C}$. If $X$ is a subshift factor of $F$ and $Y$ is a subshift factor of $X$, then $Y$ is also a factor of $Y$. Therefore it really makes sense to only consider classes $\mathcal{C}$ that are closed with respect to taking subshift factors: as noted in Section~\ref{sectPreli}, sofic subshifts form one such class.

Another class which is closed with respect to taking subshift factors is the class of coded subshifts. They can be characterized precisely as the possible factor subshifts of synchronizing subshifts~\cite{Bla86}, and in particular this class properly contains all transitive sofic subshifts. We therefore ask the following.

\begin{problem}
Are all subshift factors of $\Mul_{p/q,pq}$ coded subshifts?
\end{problem}

We guess that the answer to this problem is negative. The next result gives further evidence to support our guess.

\begin{theorem}\label{notsynch}The subshift $\trsh_{p/q}$ is not synchronizing.\end{theorem}
\begin{proof}
Assume to the contrary that $\trsh_{p/q}$ has a synchronizing word $w\neq\epsilon$. By~\cite{KK17} the CA $\Mul_{p/q,pq}$ is transitive, so $\trsh_{p/q}$ is also transitive as a factor of $\Mul_{p/q,pq}$. In particular there is a word $u\in L(\trsh_{p/q})$ such that $wuw\in L(\trsh_{p/q})$ and a word $v'\in\digs_{pq}^{2n-1}$ with $n=\abs{wuw}>1$ such that $\tr_{p/q}(y)[0,n-1]=wuw$ whenever $y\in\cyl(v',-n+1)$. Let $v=v'1$.

Let $z'=\cdots 000p.000\cdots$ and let $z_i=\sigma^i(\Mul_{p/q,pq}^i(z'))$ for $i\in\N$. By Corollary~\ref{tail} there are $s_i\in Q=\{np\mid 1\leq n<q\}\subseteq \digs_{pq}$ for every $i\in\N$ such that $z_i[0,\infty]=s_i 0^\infty$. The sequence $(z_i)_{i\in\N}$ has a converging subsequence with limit $z'\in\digs_{pq}^\Z$. By~\cite{Har12} the orbit of any aperiodic configuration is dense under the action of the additive monoid $\Z\times\N$ via the maps $\sigma^j\circ\Mul_{p/q,pq}^t$, where $(j,t)\in\Z\times\N$. In particular, there are $t\in\N$ and $j'\in\Z$ such that $\Mul_{p/q,pq}^t(z')\in\cyl(v,j')$. The configuration $z=\sigma^t(\Mul_{p/q,pq}^t(z'))$ is also a limit point of $(z_i)_{i\in\N}$ and it satisfies $z\in\cyl(v,j)$ for $j=j'-t$. Because the last letter of $v$ is $1$, it follows from $z[1,\infty]=0^\infty$ that $j\leq -2n$. Since $z$ is a limit point of $(z_i)_{i\in\N}$, it follows that $z_i\in\cyl(v,j)$ for arbitrarily large $i\in\N$. Fix some such $i$ with the additional property that $i+j+n-1\geq 0$ and let $k=i+j+n-1$. Then $i-k>0$ and $\Mul_{p/q,pq}^{i-k}(z_k)=\sigma^{-(i-k)}(z_i)\in\cyl(v,j+(i-k))=\cyl(v,-n+1)$. It follows that $\tr_{p/q}(z_k)[0,\infty]$ has a prefix of the form $sw'wuw$ for $s=s_k$ and $w'\in L(\trsh_{p/q})$.

By Lemma~\ref{construct} the inclusion $\tr_{p/q}(z)[1,n]\in W_{2,n}$ holds for all $n\in\Npos$. A compactness argument together with Lemma~\ref{wordcontexts} can be used to show that $\pre_{p/q}(\tr_{p/q}(z)[1,\infty])=Q_{q,p}(d)\cup Q_{q,p}(d+1)$ for some $d\in\digs_{pq}$. Because $w$ is synchronizing, there are $x_1,x_2\in\trsh_{p/q}$ such that $x_1[0,\infty]=dw'(wu)^\infty$ and $x_2[0,\infty]=(d+1)w'(wu)^\infty$. It follows that $x_1[1,n]\in W_{2,p}$ for every $n\in\Npos$, so by compactness and by Lemma~\ref{construct} there exists $x\in \digs_{pq}^\Z$ such that $x[0]\in Q=\{np\mid 1\leq n<q\}$, $x[i]=0$ for $i>0$ and $\tr_{p/q}(x)[1,\infty]=x_1[1,\infty]$: in particular $\tr_{p/q}(x)$ is eventually periodic, which contradicts Lemma~\ref{noperiodtrace}.
\end{proof}

\section{Realization of Complexity Functions by Sofic Subshifts}\label{secRealize}

One may ask whether the computation of the complexity function of $\trsh_{p/q}$ in Theorem~\ref{combi} would alone be sufficient to conclude that it is not sofic. The answer turns out to be negative: there is a (non-transitive) sofic subshift that has the same complexity function.

\begin{definition}
The \emph{generating function} of a function $S:\N\to\C$ is the formal power series $f(z)=\sum_{n=0}^\infty S(n)z^n$.
\end{definition}

Note that for any $m\in\Npos$ the complexity function of the regular language $\digs_m^*$ is $P_{\digs_m^*}(n)=m^n$ and its generating function is $\sum_{i=0}^\infty m^n z^n=(1-m z^n)^{-1}$.

We first compute the generating function of $P_{\trsh_{p/q}}(n)$.

\begin{theorem}
The generating function of $P_{\trsh_{p/q}}$ is $f_{p/q}(z)=\frac{1+(pq-p-q)z}{(1-pz)(1-qz)}$ for coprime $p>q>1$.
\end{theorem}
\begin{proof}
We can write $P_{\trsh_{p/q}}(n)=P_1(n)(q-1)+P_2(n)$, where $P_1(n)=(qp^n-pq^n)/(p-q)$ and $P_2(n)=p^n q$. The generating function of $P_1$ is
\[f_1(z)=\frac{q/(1-pz)-p/(1-qz)}{p-q}=\frac{q-q^2 z-p+p^2 z}{(1-pz)(1-qz)(p-q)}=\frac{(p+q)z-1}{(1-pz)(1-qz)}\]
and the generating function of $P_2$ is
\[f_2(z)=\frac{q}{1-pz}=\frac{q-q^2 z}{(1-pz)(1-qz)},\]
so the generating function of $P_{\trsh_{p/q}}$ is
\begin{flalign*}
&f_{p/q}(z)=\frac{((p+q)z-1)(q-1)+(q-q^2 z)}{(1-pz)(1-qz)} \\
&=\frac{((1-q)+(-p-q+pq+q^2)z)+(q-q^2 z)}{(1-pz)(1-qz)}=\frac{1+(pq-p-q)z}{(1-pz)(1-qz)}.
\end{flalign*}
\end{proof}

\begin{theorem}
For all coprime $p>q>1$ there is a sofic subshift $Z_{p/q}$ such that $P_{Z_{p/q}}(n)=P_{\trsh_{p/q}}(n)$ for every $n\in\N$.
\end{theorem}
\begin{proof}
To construct a sofic shift $Z_{p/q}$ whose complexity function is $P_{\trsh_{p/q}}$, we construct a factorial extendable regular language whose complexity function has the generating function $f_{p/q}$. To achieve this, first let $P,Q$ and $R$ be disjoint alphabets such that $\abs{P}=p$, $\abs{Q}=q$ and $\abs{R}=pq-p-q$. Then the regular languages $L_1=P^*$, $L_2=Q^*$ and $M=\epsilon\cup R$ have complexity functions with generating functions $(1-pz)^{-1}$, $(1-qz)^{-1}$ and $1+(pq-p-q)z$ respectively. It is easy to see that $L_1 M L_2$ is a factorial extendable regular language, so there is a sofic shift $Z_{p/q}$ such that $L(Z_{p/q})=L_1 M L_2$.
\end{proof}

\section{Conclusions}

In this paper we have studied the trace subshift $\trsh_{p/q}$ of the multiplication automaton $\Mul_{p/q,pq}$ for coprime $p>q>1$. As a result of this study we have shown that $\trsh_{p/q}$ is not sofic and not even synchronizing, i.e. $\Mul_{p/q,pq}$ is not simple according to the language theoretical classification of CA. The subshift $\trsh_{p/q}$ remains yet to be completely understood.

\begin{problem}What is the ``correct'' class of subshifts in which $\trsh_{p/q}$ belongs? Is $\trsh_{p/q}$ a coded subshift? Elements of $L(p/q)$ are easily computable by applying the CA $\Mul_{p/q,pq}$, but is there a more conceptual characterization of the language $L(p/q)$?\end{problem}

We also showed that computing the complexity function of $\trsh_{p/q}$ is not sufficient to conclude that it is not sofic. This was done by constructing a (non-transitive) sofic subshift having the same complexity function, and in fact the constructed subshift is of finite type. Mike Boyle has shown in a private communication that in the special case $p/q=3/2$ the complexity function $P_{\trsh_{p/q}}(n)$ can also be realized by a transitive sofic shift. We briefly present his construction.

Let $X$ be the sofic subshift on the alphabet $A=\{a,b,c,d,e\}$ consisting of the labels of all bi-infinite paths on the graph presented in Figure~\ref{transSofic}. The total number of paths of length $n$ (not taking the labeling into account) is equal to $4\cdot 3^n$. To count the number of possible labels for these paths, note that a sequence $w\in A^n$ is a label for multiple paths if and only if $w\in\{a,b\}^n$, and there are $2^n$ such sequences. Any such sequence $w$ is a label of precisely $4$ paths on the graph (each path starting at a different vertex). By subtracting the $3$ extra paths for each $w$ we can conclude that $P_X(n)=4\cdot 3^n-3\cdot 2^n=P_{\trsh_{3/2}}(n)$. 

\begin{problem}
For which values of $p/q$ does the exist a transitive sofic shift that has the same complexity function as $\trsh_{p/q}$?
\end{problem}

\begin{figure}
\centering
\begin{tikzpicture}[auto]
\node (d1) at (0,3) [shape=circle,draw,minimum size=6mm] {};
\node (d2) at (3,3)  [shape=circle,draw,minimum size=6mm] {};
\node (d3) at (3,0)  [shape=circle,draw,minimum size=6mm] {};
\node (d4) at (0,0)  [shape=circle,draw,minimum size=6mm] {};

\draw[-{triangle 45}] (d1) to [in=165, out=195, loop left] node {$a$} ();
\draw[-{triangle 45}] (d1) to [in=75, out=105, loop above] node {$b$} ();
\draw[-{triangle 45}] (d2) to [in=75, out=105, loop above] node {$a$} ();
\draw[-{triangle 45}] (d2) to [in=345, out=15, loop right] node {$b$} ();
\draw[-{triangle 45}] (d3) to [in=345, out=15, loop right] node {$a$} ();
\draw[-{triangle 45}] (d3) to [in=255, out=285, loop below] node {$b$} ();
\draw[-{triangle 45}] (d4) to [in=255, out=285, loop below] node {$a$} ();
\draw[-{triangle 45}] (d4) to [in=165, out=195, loop left] node {$b$} ();
\draw[-{triangle 45}] (d1) to node {$c$} (d2);
\draw[-{triangle 45}] (d2) to node {$d$} (d3);
\draw[-{triangle 45}] (d3) to node {$e$} (d4);
\draw[-{triangle 45}] (d4) to node {$f$} (d1);

\end{tikzpicture}
\caption{A transitive sofic subshift $X$ satisfying $P_X(n)=4\cdot 3^n-3\cdot 2^n$.}
\label{transSofic}
\end{figure}
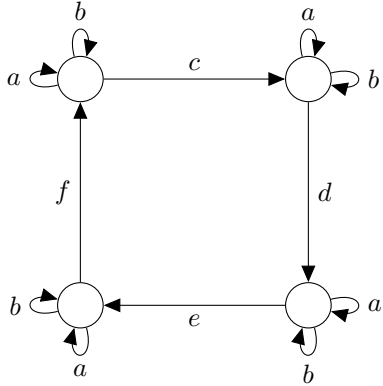

\section*{Acknowledgements}
I thank Mike Boyle for suggesting the problem considered in Section~\ref{secRealize}. The work was partially supported by the Academy of Finland grant 296018.

\bibliographystyle{plain}
\bibliography{mybib}{}

\end{document}